\documentclass[10pt]{article}
\usepackage{amsfonts,amssymb,pdfpages, xcolor}

 \usepackage{amsmath}
\usepackage{amssymb}
\usepackage{indentfirst}
\usepackage{geometry}
\usepackage{amsfonts}
\usepackage{hyperref}

\newtheorem{theorem}{Theorem}[section] 

\newtheorem{lemma}[theorem]{Lemma}
\newtheorem{proposition}[theorem]{Proposition}
\newtheorem{corollary}[theorem]{Corollary}

\newtheorem{definition}[theorem]{Definition}

\newtheorem{remark}[theorem]{Remark}

\newcommand{\N}{{\mathbb{N}}}   

\newcommand{\R}{{\mathbb{R}}}   
\newcommand{\Q}{{\mathbb{Q}}}   
\newcommand{\C}{{\mathbb{C}}}   

\newcommand{\OK}{{\overline{{\mathbb K}}}}
\newcommand{\OR}{{\overline{{\mathbb R}}}}

\newcommand{\A}{{\mathcal{A}}}

\newcommand{\qed}{\enspace\vrule  height6pt  width4pt  depth2pt}
\newenvironment{proof}{\par\noindent{\bf Proof.}}{$\qed$\par\bigskip}

\begin{document}

\title{Transition Probabilities and Almost Periodic Functions}
  
\author{   Juriaans, S. O.   \and  Queiroz, P.C.\footnote{IME-USP and UFMA, ostanley@usp.br, pqueiroz@ime.usp.br}}

\date{}

\maketitle

\begin{abstract} We study transition probabilities of generalized functions as introduced by Colombeau and Gsponer.  We   formally introduz the study of H. Bohr almost periodic functions in the generalized context and use them to give exact values of transition probabilities in dimension $\ 2$. We also  prove the existence of transition probabilities,  in the generalized sense,    for any  moderate net $\ T = (T_{\varepsilon})$ with $\ T_{\varepsilon} : \mathbb{H}\longrightarrow   \mathbb{H}$,  $\mathbb{H}$ a Hilbert space.  In particular,   we consider  the case of selfadjoint Hilbert-Schmidt operators. 

This alludes to the possibility of existence of transition probabilities  in the  Fock space, in the sense introduced by Colombeau and Gsponer,    if the spectrum of the operators involved are pure infinities or infinitesimals. This was already indicated in  two   papers  coauthored by J.  Aragona and  J.F. Colombeau  et. al.
  
   \end{abstract}

\section{Introduction}

Schwartz's Theory of Distributions is one of his major contribution to the Theory of Generalized Functions. It marked major advances in the field of P.D.E.'s. Being  a linear environment, it does not  deal with problems involving products of distributions. Although there still exist the  believe that Schwartz's Impossibility Theorem alludes to the non-existence of an associative diferencial  milieu in which these products can be dealt with, this is certainly not the case. In fact, such environments do exists and have more than proven their consistency and importance (see \cite{Colom1, Colom2, Colom3, Colom4, Colom5, Colom6, Colom7, Colom8, Colom9, gkos1, gkos2, gkos3, gkos4, hosk, ku, ku1, ku2, ku3, ku4, mog, Ober, ros1, ros2, ros3, ros4, ros5}).    It was in  the eighties that  Rossinger and Colombeau undertook the challenge and started to  develop  non-linear theories in which the  multiplication of distributions make sense. Important researchers such as Aragona, Biagioni, Egorov,  Grosser,  Kunzinger, Oberguggenberger,  Pilipovic,  Scarpal\'{e}zos, Steinbauer  and  Vickers, to name a few,  made possible the development of  the algebraic and topological theory of these new generalized environments.    The algebraic developments were key in permitting to view the real potential of the generalized milieu. In particular, the boolean  algebra  $\ {\mathcal{B}}(\overline{\mathbb{K}})$ of  $\ \overline{\mathbb{K}}$ is key in this development together with the group of invertible elements $\ Inv(\overline{\mathbb{K}})$. The latter is open and dense in $\ \overline{\mathbb{K}}$ and its complement consists of zero divisors each of which has a nontrivial element of $\ {\mathcal{B}}(\overline{\mathbb{K}})$  in its annihilator, i.e., given a non unit $\ x \in \overline{\mathbb{K}}$ there exists a nontrivial  idempotent $\ e\in {\mathcal{B}}(\overline{\mathbb{K}})$ such that $e\cdot x =0$. It also holds that an element $\ x\in \overline{\mathbb{K}}$ is a unit if and only if  there exists $\ n\in \mathbb{N}$ such that $ \ x\geq \alpha^n$, where $\ \alpha =[\varepsilon\longrightarrow \varepsilon]$ is the natural, or standard,  gauge of $\ \overline{\mathbb{K}}$. Each idempotent is uniquely linked to a certain kind of subset in the parameter space defining these generalized milieu. The idempotents  are crucial and important in the definition of the notion of interleaving (see \cite{OV}), the existence of roots for generalized polinomial functions (see \cite{horman1, ver2}) and the notion of the support and the existence of an atlas for a  generalized differential manifolds (see \cite{juroliv}) as we shall see next. The connection with non standard Analysis (see \cite{rob}) is clear but in a very different way than expected (see \cite{walter}).

These   developments led to the introduction of a generalized milieu  which extends in a natural way the classical Newtonian Calculus. This is called {\it Generalized Differential Calculus} (see \cite{AraFerJu, OJRO, afj2, afj1, AJ, juroliv, inter}). Many Classical results do hold in this generalized environment. For example, one has a Generalized Fixed Point Theorem which is useful when proving existence of solution for equations in the the generalized environment (see \cite{juroliv}).

  Generalized Differential Calculus is extended in the following way. Given a classical differential manifold $\ M$ there exist a Generalized Differential  Manifold $\ M^*$ into which $ \ M$ embeds discretely. Problems involving distributions on $ \ M$ translates naturally into questions involving $\ C^{\infty}$ functions on $\ M^*$. These new environments are as natural as the classical ones and several classical results have an analogue in  the general environment making it possible to prove most results intrinsically. But they have a feature which exists in the classical environments: infinitesimals and infinities coexists in these milieu and effectively effects physical reality  although going undetected by this physical reality (see \cite{juroliv}).
  
  Another  important notion,  introduce recently in \cite{juroliv}, is the notion of {\it supports of generalized objects}. For genereralized functions there  is such a notion but it refers more to the domains of the functions. Here the idea is to look at compactly supported elements or elements which, when multiplied by an idempotent, becomes compactly supported and are in the halo of point of physical reality.

Colombeau and Gsponer, see \cite{Colom8, sponer1}, introduced the notion  of transition probabilities in their attempt to establish a theory to deal with certain problems in QFT.   Since then, transition probabilities were studied in   \cite{Varna,  transition, sponer1,  sponer} among to carry out the blue print layout in   \cite{Colom8, sponer1}.  Hopefully this will turnout to be another major contribution of J.F. Colombeau in the understanding of physical reality.  In \cite{unpublish, transition} progress was made and in \cite{transition} some questions raised remained unanswered. The existence of transition probabilities of $\ 2\times 2$ matrices were simulated using MatLab indicating their existence and making progress in the pursue of the ideas of Colombeau and Gsponer. In \cite{unpublish} light is shed on the general case showing that it is feasible to carry out the proposed blue print mentioned before. In this manuscript, we first generalized the notion of transition probability introducing the {\it generalized transition probability} and proving the existence of the latter for any symmetric matrix in $\ M_n(\overline{\mathbb{C}})$. We first prove this for elements of $\ M_n(\overline{\mathbb{C}}_{as})$ and then go on to prove this in the general case. We revisit the examples given in \cite{transition} and give explicit calculation of the values involved. We also look at compactly supported symmetric operators acting on Hilbert spaces and establish results analogous  to the finite dimensional case.   

The manuscript  is organized   as follows.  In the next section,  the basics  of the Theory of Generalized Functions and Generalized Differential Calculus  is recalled to give the reader a good overview  of the status quo. Notation is recalled and established aiming to  facilitate  the reading of the manuscript. Next,  the basics of almost periodic functions is collected and  extended   to the generalized environment. In particular,  an a.p.f. may become   a periodic function in the generalized environment. In the fourth section,    the basics of Fock Space and the generalized context  constructed by Colombeau and Gsponer are recalled. The  {\it generalized transition probability}  is defined having the original definition as a particular case. In section 5,   generalized transition probabilities of finite matrices is studied. The focus is on compactly supported matrices.  It is  proved that in this case  the generalized transition probability always exists and its   support is calculated.   The infinite dimensional case is considered, in particular,   Hilbert-Schmidt operators and compactly supported self-adjoint operators are considered. In the last section,  the examples given in \cite{Varna, transition}  are revisited and  the exact value of all the transition probabilities are given. A    satisfactory solution is given in the finite dimensional case.  The infinite dimensional case is revisited and some results about the generalized transition probability are proved in this case.


\section{Preliminaries}

We refer the reader to the following excellent texts  on the Theory of Colombeau Generalized Functions: \cite{AraBia,  Colom1, Colom2, Colom3,   Colom5, Colom6, Colom7, Colom8, Egorov, ku, Nedel, mog, Ober, OV,  S2, tod, todver, V1, V2, V3}. For the algebraic part of the  theory we refer the reader  to \cite{agj, AJ, JJO, horman1, ku, ver, ver2}.  The reader interested in the foundations of the Generalized Differential Calculus we refer to \cite{AraFerJu, OJRO, juroliv}. The topological foundation of the theory can be found in \cite{afj2, afj1, AJ, Bia, S1, S2, S3}. For applications,    more on the theory and physics we refer the reader to \cite{antos, agj, agj1,  AS,  chris, spjornal,  walter1,   gkos1, gkos2, gkos3, gkos4, hosk, kastler, wolfram}.  Here we shall recall part of the theory in the form and sequel that we shall be using it in this paper.

Given a field $\ \mathbb{K}\in \{\mathbb{R},\mathbb{C}\}$ consider the product space of nets  $\ \mathbb{K}^I$, where $\ I=]0,1]$.   We call $\ \alpha =[\varepsilon\longrightarrow \varepsilon]$  the natural gauge and  if $x,y\in  \ \mathbb{R}^I$ we write $\ x\ \leq y$ iff there exists $\ \eta>0\in I$ such that $x(\varepsilon)\leq y(\varepsilon), \ \forall\ \varepsilon\in ]0,\eta]$. The set of moderate nets of  $\ \mathbb{K}^I$  is $\ \mathcal{E}_M(\mathbb{K}) =\{x\in  \ \mathbb{K}^I\ :\ $ there exists $\ n\in \mathbb{Z}$ such that $ |x|\leq \alpha^n\}$, where $|x|(\varepsilon)=|x(\varepsilon)|$.  It is clear that $\ \mathcal{E}_M(\mathbb{K})$ is a subring of $\ \mathbb{K}^I$ and that $\mathcal{N}(\mathbb{K})$ $=\{x\in  \mathcal{E}_M(\mathbb{K})\ : |x|\leq \alpha^q, \forall q\in \mathbb{N}\}$ is an ideal of this subring. The Colombeau ring of generalized numbers over $\ \mathbb{K}$ is the quotient ring $\ \overline{\mathbb{K}}= \mathcal{E}_M(\mathbb{K})/\mathcal{N}(\mathbb{K})$. It holds that $\ \overline{\mathbb{C}}= \ \overline{\mathbb{R}}+i\ \overline{\mathbb{R}}$, $i^2=-1$. The ring of Colombeau Generalized real numbers $\ \overline{\mathbb{R}}$ will be the basic under lying structure for our generalized differential Calculus, Geometry and Analysis. It is an ultrametric topological ring. This  topology is defined in \cite{Bia, S1, S2, S3} and was coined the {\it  sharp topology}. The norm of an element $\ x\in \overline{\mathbb{R}}$ is $\ \|x\|=e^{-V(x)}$, where $\ V(x)$ is the supremum  of the set $\ \{r\in \mathbb{R}\ :\ |x|\leq \alpha^r\}$.  It is interesting to note that the classical equality $\ H^2 =H$ does not hold  in the generalised milieu, i.e., the Heaviside Function $\ H$ is not an idempotent! The latter is a statement about the perception of our physical reality.

  The algebraic properties of this ring were first investigated in \cite{AJ}. For examples, all maximal ideals of $\ \overline{\mathbb{R}}$ where described, the boolean algebra $\mathcal{B}(\overline{\mathbb{R}})$  was described and proved that   its group of invertible elements $\ Inv(\overline{\mathbb{R}})$  is open and dense in the sharp topology. The elements of $\mathcal{B}(\overline{\mathbb{R}})$ are elegantly linked  to subsets of the parameter space $\ I$ (see  \cite{AJ, walter}). Moreover, if $\ 0\not = x \in \overline{\mathbb{K}}$ then there exists $\  e\in \mathcal{B}(\overline{\mathbb{R}})$ such that $ \ ex\in Inv(e\overline{\mathbb{R}})$. Note  that if   $\ e,f\in \mathcal{B}(\overline{\mathbb{R}})$ are such that $f<e$ then $\ ef =f$. An excellent reference is  \cite{ver} (see also \cite{JJO, walter}).

  In \cite{OJRO, afj2, afj1} a basis for this topology is given that is compatible with the algebraic structure. A basic neighbourhood is $\ V_r(0)=\{ x\in \overline{\mathbb{R}}\ : \ |x|< \alpha^r\}$, $ r>0$. An $\ x\in V_1(0)$ is such that $\ |x|<s,\forall\ s\in \mathbb{R}^+$. On the other hand, $\alpha^{-1}>s, \forall \ s\in \mathbb{R}$. Consequently, in this new milieu infinitesimals and infinities coexist and can cancel out when multiplied.

Given $\Omega\subset \mathbb{K}^n$, let  $\widetilde{\Omega}_c  = \{  x= (x_1 , \cdots ,x_n) \in \overline{\mathbb{K}}^n \ : \ \exists \  \eta >0 , \{ \hat{x}(\epsilon) \ :\ \epsilon <\eta\} \subset\subset \Omega\}$, where $\ \subset\subset$ means that the set has compact topological closure. One can prove that $\ \Omega $ is discretely embedded in $\ \widetilde{\Omega}_c  $. This is one of the most important definitions in the recent developments of the Theory and was given by Kuzinger-Oberguggenberger (see \cite{ku, ko}). The importance of this definition is the following results proved by them: {\it A Colombeau Generalized Function is zero if and only if its point values on $\ \widetilde{\Omega}_c $ are all zero in $\ \overline{\mathbb{R}}$} (see \cite{ku, ko}). This actually translate into having a natural  candidate for a domain for a generalized functions defined on $\ \Omega$. Together with the algebraic developments, this is one of the stepping stones for our Generalized Differential Calculus (see \cite{AraFerJu}).

  A function  $f\in \mathcal{F}(\widetilde{\Omega}_c, \overline{\mathbb{K}})$ is said to be differentiable   if there exists $z_0\in \overline{\mathbb{K}}$ such that  $$\lim_{x\to x_0}\frac{f(x)-f(x_0)-z_0(x-x_0)}{\alpha_{-\log(\|x-x_0\|)}}=0$$ 
  
  \noindent Denoting this limit by $f^{\prime}(x_0 )$,  one shows that this is a derivation on   $\mathcal{F}(\widetilde{\Omega}_c, \overline{\mathbb{K}})$  which satisfies almost all properties of classical differential calculus and we have the following embedding theorem. This theorem is a complete solution to the apparent Schwartz Impossibility Paradox in the sense of Newtonian  Calculus.

\begin{theorem} [Aragona-Fernadez-Juriaans]  There exists a continuous    $\ \mathbb{K}-$ linear injection   
$\kappa:D^{\prime}\to C^\infty(\widetilde{\Omega}_c,\overline{\mathbb{K}})$   such that    $\kappa\Big(\frac{\partial f}{\partial x_i}\Big)= \frac{\partial(\kappa(f))}{\partial x_i}, ~\forall~ f \in D^{\prime}, \forall i$.
\end{theorem}       

Generalized Calculus is further developed in \cite{OJRO} and is the basis for the Generalized Differential Geometry and a Fixed Point Theorem in this generalized milieu (see \cite{juroliv}).

The subset $\ \OK_0 := \{x\in \OK \ | \  x\approx 0\} $ is a subring of $\OK$ containing  the open ball $B_1(0)$ properly (see \cite[Proposition 2.1]{AJ}).  More generally, $\OK_{as} := \{x\in \OK \ | \  x\approx x_0\in \mathbb{R} \} $ is a subring of $\OK$ containing the open ball $B_1(0)$. In this paper we consider matrices over $\OK_{as}$ and consider also the case when $\  \lim\limits_{\varepsilon \longrightarrow 0}\hat{x}(\varepsilon) =\infty$,  coining the name {\it pure infinities} to such elements.  Our standard reference for matrices over $\OK$ will be  \cite{horman1, ver2}. It should be noted however that this subject was first studied in \cite{mayer, ver2}. The notion of eigenvalues and eigenvectors is as defined in  \cite{horman1}. Examples in \cite{horman1} show that one must be careful in the generalized environment when defining these notions.  For the readers sake,  we recall some  basics needed in  the sequel.   

A vector $(v_1, \cdots ,v_n) =v\in \OK^n$ is a free-vector if $\ \ span_{\OK}[v_1, \cdots ,v_n] = \OK \ $ and a generalized number $\lambda \in \OK$ is an eigenvalue of an $n\times n$ matrix $A\in M_n(\OK)$ if there exists a free-vector $v\in \OK^n$ such that $A\cdot v = \lambda v$. In this case  one has that $det(\lambda I - A)=0\in\OK$,    $Ker(\lambda I - A)$ is non-trivial  and contains a free vector.  Note that $v\in \OK^n$ is a free-vector if and only if $\|v\| \in Inv(\OR)$.

  We now recall the definition of the support of a generalized object given in \cite{juroliv} (see \cite{juroliv} for more details). It can be defined in any generalized mulieu, in special if the nets involved are membrane like (see \cite{OJRO} for the definition of a membrane). For  $p=[(p_{\varepsilon})]\in\overline{\mathbb{C}}^n$,  consider   the set $\{q\in \mathbb{C}^n\ : \exists\  \varepsilon_n\rightarrow 0,  \ p_{\varepsilon_n}\rightarrow q\}$.  Algebraically this can be written as:  Given  $q_0\in \mathbb{C}^n$, we have that   $\ q_0\in \{q\in \mathbb{C}^n\ : \exists\  \varepsilon_n\rightarrow 0,  \ p_{\varepsilon_n}\rightarrow q\}$ if and only if there exists $e\in {\cal{B}}(\overline{\mathbb{R}})$ such that $e\cdot p\approx e\cdot q_0$ (extending the notion of association to $\overline{\mathbb{K}}^n$ in the obvious way). This is a compact subset of $\ \mathbb{C}^n$ to which  we  shall refer to  as the {\it{support of the point}} $p$ and denote  it  by $supp(p)$. It follows that there exists a complete set of orthogonal  idempotents  $(e_{\lambda})$ such that 

$$p\ =\sum\limits_{x_{\lambda}\in supp(p)}e_{\lambda}\cdot p  $$ $$ (e_{\lambda}\cdot p \approx e_{\lambda}\cdot x_{\lambda})$$

This definition can be extended to any generalized milieu and we shall do so here in this manuscript. It can obviously be extended to include $\ \infty$ in the case of generalized numbers.  A generalized number $\ \lambda$ is an {\it infinity} if $\ \infty \in supp(\lambda)$ and is a {\it pure infinity} if $\ supp(\lambda)=\{\infty\}$.

The notion of interleaving given by Oberguggenberg-Vernaeve (see \cite{OV}) once again shows the importance of the elements of $\ \mathcal{B}(\overline{\mathbb{K}})$. The interested reader can read more about this in several sections of \cite{juroliv}. This notion plays an important role in defining a generalized manifold starting with a classical manifold and embedding the latter discretely into the former. For the notion of {\it membrane} and its generalization, {\it  internal sets}, we refer the reader to \cite{OJRO, OV}.


\section{Almost Periodic Functions}

In this section we shall prove the existence of mean values for some functions. It is here also that we link the transition probabilities with the mean value of a almost periodic function.   We start recalling the concept of {\it almost periodic  function}.  The basic references we used are \cite{Besi, Bohr}. This remarkable theory seems to be have conceived to merge naturally  into the Theory of Generalized Differential Calculus.  One of the main things of this section is the fact that an almost periodic function becomes a periodic function in the generalized context.  After  recalling the definition of an almost periodic function we  prove some useful ways to calcule their mean value and then show their link to transition probabilities.

A set of real numbers $E\subset \R$ is said to be  {\it  relatively dense} in $\ \mathbb{R}$  if there exists $L > 0$ such that for each $a \in \R$ one has that $] a, a+L [ \cap E \neq \emptyset$. Roughly speaking,  this means that there are no arbitrary large gaps among the elements of $\ E$.

\begin{definition}
Given a complex valued function $f: \R \longrightarrow \C$ and $\varepsilon > 0$,  the real number $\tau =\tau(\varepsilon) =\tau_f(\varepsilon)  > 0$  is called a translation number of $f$ corresponding to $\varepsilon$  if $\ |f(x +\tau) - f(x)|\leq \varepsilon, \forall x\in \R$.       A continuous function $f:\R \longrightarrow \C$ is said to be {\it almost periodic} ( an a.p.f.)  if  for each $\varepsilon>0$ the set  $\{ \tau \in \R \ | \tau = \tau_f(\varepsilon)\}$ is relatively dense.
\end{definition}

Consider the $\C$-vector space $V = Span_{\C}[ e^{i\lambda x}\ | \lambda \in \R]$ of trigonometric polynomials or pure vibrations which include the harmonic vibrations which are the generating set of the Wiener algebra.   The elements of $V$ are bounded functions and thus we may consider the normed space $(V, ||\ ||_{\infty})$.  Let $\overline{V}$ be its closure in $L_{\infty}(\R)$. Thus $\overline{V}$ consists of uniform limits of elements of $\ V$.  Note that the elements of $\ V$ are periodic functions and thus the set of periodic functions is dense in $\overline{V}$. 

With an {\it elementary step function} we mean the characteristic function of an interval $I\subset \R$.   A {\it step function} is a finite $\R$-linear combination of elementary step functions.  Since one clearly has that   $M(f( x + c)) = M(f)$, for any   constant $c\in \R$ and $M(f(\lambda x)) = M(f)$, for any $\lambda > 0$,  we way consider elementary step functions defined on an interval of the form $[0, \beta]$ with $\beta <1$.   We say that $f$ is an elementary periodic step function if $f$ has period $1$  and on $[0, 1]$ $f$ is the characteristic function of an interval $[0, \beta]$.

If $\  f$ is periodic of period $\ L$ then we may approximate $\ f$ uniformly by a sequence of step functions on $\ [0, L]$. Hence  $\ f $ can be approximated uniformly  by  periodic  step functions on $\ \R$.  It follows that an a.p.f. is a uniform limit of step functions on $\R$. So if we denote by $ S$ the set of periodic step functions then we have that $\overline{V}\subset \overline{S}$.  If $\ A\in M_n({\overline{\mathbb{C}}})$ is a hermitian matrix then $\  exp(iA)$ is a   unitary  matrix and hence its   eigenvalues   are of the form $\ exp(i\theta),  \theta \in Spec(A)\subset  {\overline{\mathbb{R}}}$. Since we are interested in transition probabilities, which include such expressions, we look at the general situation.  Let $[f_{\varepsilon}]=f\in \overline{\mathbb{R}}$  be integrable when seen as a function $\ f\  : \ I\longrightarrow \mathbb{R}$. Define the {\it generalized transition probability}  of $\ f$ to be 

$$ \nu(f) = \frac{1}{\alpha}\int\limits_0^\alpha f(t)dt =\bigg[\varepsilon \longrightarrow \frac{1}{\varepsilon}\int\limits_0^{\varepsilon}f(t)dt\bigg]$$

\noindent   In particular, if$\ f\approx f_0\in \mathbb{R}$ then $\nu(f) = f_0$, as can be easily seen. It will also follow from a result of this section.  This particular case  explains our interest in almost periodic functions and their mean value  and thus  linking  all sections of this manuscript with the concept of {\i transition probability} defined by Colombeau-Gsponer in \cite{Colom8, sponer1}.  The results from this section  show that this definition makes sense.   When a matrix or operator $\ A$  on a Hilbert space  is involved, we shall write $\ \nu(A)=\nu(A,u,v)=\nu(|\langle u\mid Av\rangle|)$ without  specifying  the dependence of  the pair of fixed orthogonal vectors involved in the definition. The same will be done when transition probabilities of fixed orthogonal pairs are involved (see the next section).   This definition should be compared with the definition of $\ \mu(e), e\in \mathcal{B}(\overline{\mathbb{K}})$ given in \cite{juroliv}. Since a natural way of approximating  discrete data is the use of elementary step functions, it is natural to think of idempotents stemming from physical reality as resulting from such elementary step functions. Hence their probability is well defined. See \cite{juroliv} for more information with regard to   $\ \mu(e), e\in \mathcal{B}(\overline{\mathbb{K}})$.    

We recall  the following result fundamental results concerning   Bohr's a.p.f (see \cite{Besi, Bohr}).

\begin{theorem} [Fundamental Theorem of A.P.F.]
\label{fund}
A complex valued function $\ f\ $ is almost periodic if and only if $\ f\in \overline{V}$. Moreover,  $\overline{V}$ is a $\C$-algebra and its elements are  bounded uniformly continuous functions. 
\end{theorem}

If $f $ is a periodic function  with period $L$,  then it  is easy to prove  that  $$M(f)  : = \lim\limits_{T\longrightarrow \infty} \frac{1}{T}\int\limits_{0}^{T}f(x)dx =   \frac{1}{L} \int  _{0}^{L} f(x)dx$$ 
\vspace{.3cm}

\noindent In particular, $M(|sin(Ax)|)=M(|cos(Ax)|)=\frac{2}{\pi}$, $\forall\  A>0$.    We shall use this in the section of examples.

 \begin{theorem} [Mean Value Theorem]
\label{meanval}

For every a.p.f.  the mean value  $$M(f) : = \lim\limits_{T\longrightarrow \infty} \frac{1}{T}\int\limits_{0}^{T}f(x)dx $$

\noindent exist.

\end{theorem}

\noindent Hence,  an  a.p.f.  may not be periodic but  it still  has a mean value.       Theorem~\ref{meanval}  is the foundation for a theory of Fourier series of almost periodic functions.  In fact, given an a.p.f. $\ f$ there exist an at most an  enumerable set of vibrations $\ e^{-i\lambda x}$ such that $\ M(f(x)e^{-i\lambda x})\neq 0$ and $\ f\equiv 0$ if and only if $\ M(|f|^2)=0$. Parseval's Equation also holds:

$$M(|f(x)|^2)=\sum\limits_{M(f(x)e^{-i\lambda x})\neq 0}|M(f(x)e^{-i\lambda x})|^2$$

   We proceed to  prove  alternative formulas for the mean value of an a.p.f.  and by composing them with pure infinities, we introduce their generalized transition probability int   the generalized environment.

\begin{lemma}
\label{stepf}

Let $f$ be an elementary periodic step function and $\gamma > 0$.  Then $$M(f) = \lim\limits_{T\longrightarrow \infty} \gamma \cdot T^{\gamma}\cdot \int \limits_T^{\infty} \frac{ f(x)}{x^{1+\gamma}}dx$$

\end{lemma}

\begin{proof} {\bf Case $\gamma = 1$:}  We first suppose that $T = n$ is an integer and hence we may calculate its limit as $\lim\limits_{n\longrightarrow \infty} n\cdot \int \limits_n^{\infty} \frac{ f(x)}{x^2}dx$ $=  \lim  \limits_{n \longrightarrow \infty} n\cdot (\sum\limits_{k = n}^{\infty} [\frac{1}{k} - \frac{1}{k+\beta}])$ $= \beta \cdot (    \lim  \limits_{n \longrightarrow \infty} n\cdot (\sum\limits_{k = n}^{\infty} \frac{1}{k(k+\beta)}))$. Note that  $\int\limits_n^{\infty}\frac{1}{u(u + \beta)}du < \sum\limits_{k = n}^{\infty} \frac{1}{k(k+\beta)} < \int\limits_{n - 1}^{\infty}\frac{1}{u(u + \beta)}du$ $= \int\limits_{n - 1}^{n}\frac{1}{u(u + \beta)}du+\int\limits_{n}^{\infty}\frac{1}{u(u + \beta)}du$. Applying L'Hospital's rule from elementary calculus we get that  $\lim \limits_{n\longrightarrow \infty} n\cdot ( \int\limits_{n - 1}^{n}\frac{1}{u(u + \beta)}du)=0$. From this it follows that   $\lim \limits_{n\longrightarrow \infty}n\cdot ( \sum\limits_{k = n}^{\infty} \frac{1}{k(k+\beta)})  = \lim \limits_{n\longrightarrow \infty} n \cdot (\int\limits_n^{\infty}\frac{1}{u(u + \beta)}du) =  \lim \limits_{x\longrightarrow \infty} x \cdot (\int\limits_x^{\infty}\frac{1}{u(u + \beta)}du)$ $=\lim \limits_{x\longrightarrow \infty} \frac{x}{\beta}\cdot \ln (\frac{x+\beta}{x}) =1$, where we once again applied L'Hospital's rule to get the last equality.  

 Now if $T\in \R$ then $T = n +r$, with $0\leq r\leq 1$. We have that  $\lim\limits_{T\longrightarrow \infty} T\cdot \int \limits_T^{\infty} \frac{ f(x)}{x^2}dx$ $= \lim\limits_{T\longrightarrow \infty} T\cdot \int \limits_{n+r}^{\infty} \frac{ f(x)}{x^2}dx= \lim\limits_{T\longrightarrow \infty} n\cdot \int \limits_n^{\infty} \frac{ f(x)}{x^2}dx + \lim\limits_{T\longrightarrow \infty} r\cdot \int \limits_n^{\infty} \frac{ f(x)}{x^2}dx+ \lim\limits_{T\longrightarrow \infty} (n+r)\cdot \int \limits_n^{n+r} \frac{ f(x)}{x^2}dx$.  The integral in the middle goes to $0$, as $T$ goes to infinity  (independ of the value of $r$),  and the last integral is $\leq (n+r)[\frac{1}{n}-\frac{1}{n+r}] = \frac{r}{n}$.  From this the result follows. 

\noindent {\bf Case $\gamma \neq 1$:}  We will suppose first that $\gamma \in \Q$ and $T = n \in \N$.  As before, we have  \\  $\lim\limits_{n\longrightarrow \infty} \gamma \cdot  n^{\gamma}\cdot \int \limits_n^{\infty} \frac{ f(x)}{x^{1+\gamma}}dx$ $=  \lim  \limits_{n \longrightarrow \infty} n^{\gamma}\cdot (\sum\limits_{k = n}^{\infty} [\frac{1}{k^{\gamma}} - \frac{1}{(k+\beta)^{\gamma}}])$ $=      \lim  \limits_{x \longrightarrow \infty} x^{\gamma}\cdot (\int\limits_{x}^{\infty} [ \frac{1}{u^{\gamma}} - \frac{1}{ (u+\beta)^{\gamma}}]du)$ \\  $=  \frac{1}{1-\gamma}  \lim  \limits_{x \longrightarrow \infty} x^{\gamma}\cdot ((u+\beta)^{1-\gamma} - u^{1-\gamma}) $.  If $ 1 -\gamma > 0$ we set $ 1- \gamma =\frac{p}{q}$ and otherwise we set $ \gamma - 1 =\frac{p}{q}$. In the first case  the limit becomes   $ \frac{1}{1-\gamma}  \lim  \limits_{x \longrightarrow \infty} x^{\gamma}\cdot (\sqrt[q]{(u+\beta)^p} - \sqrt[q]{u^p} )$ and in the latter  case it becomes  $ \frac{1}{1-\gamma}  \lim  \limits_{x \longrightarrow \infty} x^{\gamma}\cdot  (\frac{1}{\sqrt[q]{(u+\beta)^p}} - \frac{1}{\sqrt[q]{u^p}} )$. Rationalizing this expression in the usual way,  and calculating its limit we obtain that it equals $\ \beta$. By continuity, the result holds for all $\ \gamma > 0$.  
 \end{proof}

\begin{corollary}
\label{stepf1}

Let $f$ be a uniform limit of periodic step functions and $c > 0, \gamma > 0$.  Then  $$M(f) = \lim\limits_{T\longrightarrow \infty}\gamma \cdot  T^{\gamma}\cdot \int \limits_T^{\infty} \frac{ f(x)}{x^{1+\gamma}}dx  = \lim\limits_{\eta\longrightarrow 0} \frac{1}{\eta} \int \limits_0^{\eta} f\bigg(\frac{1}{\epsilon^c}\bigg)d\epsilon$$

In particular,  $\nu(f\circ \alpha^{-c})=M(f)$.
 \end{corollary}

\begin{proof} If $f$ is the uniform limit of the sequence $(f_n)$ then clearly we have that  $ \lim\limits_{n\longrightarrow \infty} M( f_n)   = M(f)     $.  To prove the second part,  make the chance of variable $ x = \frac{1}{\varepsilon^c}$ obtaining that $\lim\limits_{\eta\longrightarrow 0} \frac{1}{\eta} \int \limits_0^{\eta} f(\frac{1}{\epsilon^c})d\epsilon = \lim\limits_{T\longrightarrow \infty} \gamma \cdot  T^{\gamma}\cdot \int \limits_T^{\infty} \frac{ f(x)}{x^{1+\gamma}}dx$, where $\gamma =\frac{1}{c}$. The result now follows from Lemma~\ref{stepf}. \end{proof}





\begin{theorem}
\label{meanval1}

Let $\ f$ be an a.p.f. $\gamma > 0 $ and $\lambda : [0, 1]\longrightarrow [0, \infty]\ $ be continuos with $\lambda (0) = 0$ an let $\nu$ be its inverse. Suppose that $\nu \in \mathcal{C}^{k}([0, \lambda (1)])$,  with  $k \geq 2$ , is  not a null  function. Then   $$M(f)  =  \lim\limits_{T\longrightarrow \infty} \gamma\cdot T^{\gamma}\cdot \int \limits_T^{\infty} \frac{ f(x)}{x^{1+\gamma}}dx  =  \lim\limits_{\eta\longrightarrow 0} \frac{1}{\eta} \int \limits_0^{\eta} f\bigg(\frac{1}{\lambda(\epsilon)}\bigg)d\epsilon$$

In particular, $ \nu(f\circ\frac{1}{\lambda})=M(f)$.

\end{theorem} 

\begin{proof}  As seen above,   $\ f$ can be uniformly approximated by periodic  steps functions, and hence the first equality follows from Corollary~\ref{stepf1}. To prove the second equality, note that since $\nu$ is not  a null function there exists  $ q \in \N^*$ such that $ b := \nu^{(q)}(0)\neq 0$ and $\nu^{(j)}(0)=0$, $\forall \  0 < j < q$. Making  the chance of variable $ x= \frac{1}{\lambda(\epsilon)}$  one obtains  $$ \lim\limits_{\eta\longrightarrow 0} \frac{1}{\eta} \int \limits_0^{\eta} f\bigg(\frac{1}{\lambda(\epsilon)}\bigg)d\epsilon =  \lim\limits_{T\longrightarrow \infty}\frac{1}{\nu(\frac{1}{T})}\int \limits_T^{\infty} \frac{f(x)\nu^{\prime}(\frac{1}{x})dx}{ x^2}$$

From hypothesis we have that  $\lim\limits_{T\longrightarrow \infty} T^q \nu(\frac{1}{T}) =\frac{b}{q!}$ and $\ \nu^{\prime} (\frac{1}{x}) = \frac{b}{q!x^{q-1}} + g(x)$, $|g(x)| \leq \frac{C}{x^{q}}$, where $C > 0$ is a constant.  Hence we have that   $$  \lim\limits_{T\longrightarrow \infty}\frac{1}{\nu(\frac{1}{T})}\int \limits_T^{\infty} \frac{f(x)\nu^{\prime}(\frac{1}{x})dx}{ x^2}  =  \lim\limits_{T\longrightarrow \infty}\frac{T^q}{T^q\nu(\frac{1}{T})}\int \limits_T^{\infty} \frac{f(x)[\frac{b}{q!x^{q-1}} + g(x)]}{ x^2}dx =$$ $$ \lim\limits_{T\longrightarrow \infty}\frac{T^q}{b/q!}\int \limits_T^{\infty} \frac{f(x)[\frac{b}{(q - 1)!x^{q-1}} + g(x)]}{ x^2}dx = \lim\limits_{T\longrightarrow \infty}q\cdot T^q \cdot \int \limits_T^{\infty} \frac{f(x)dx}{x^{q+1}} +  \lim\limits_{T\longrightarrow \infty}T^q \cdot \int \limits_T^{\infty} \frac{q!f(x)g(x)dx}{bx^{2}}  $$

Since $\ f$ is bounded we have, from our hypothesis,   that $$\Big |  \lim\limits_{T\longrightarrow \infty}T^q \cdot \int \limits_T^{\infty} \frac{q!f(x)g(x)dx}{bx^{2}}  \Big | \leq   \frac{q!C\|f\|_{\infty}}{b} \lim\limits_{T\longrightarrow \infty}T^q \cdot \int \limits_T^{\infty}\frac{dx}{x^{q+2}} = 0 $$

It follows that $  \lim\limits_{T\longrightarrow \infty}\frac{1}{\nu(\frac{1}{T})}\int \limits_T^{\infty} \frac{f(x)\nu^{\prime}(\frac{1}{x})dx}{ x^2}  =  \lim\limits_{T\longrightarrow \infty}q\cdot T^q \cdot \int \limits_T^{\infty} \frac{f(x)dx}{x^{q+1}}  = M(f)$.\end{proof}

 Note that if $\nu$ is a null function then $\  \lim\limits_{\eta\longrightarrow 0} \frac{1}{\eta} \int \limits_0^{\eta} f(\frac{1}{\lambda(\epsilon)})d\epsilon = 0$. This is easily seen since, in this case,  we have that $|\nu^{\prime} (x)| \leq C(q)\frac{1}{x^q}\ $, where $C(q)$ is a constant,  for all $\ q > 0$.
 \vspace{0.5cm}

 $ \nu(f\circ\frac{1}{\lambda})=M(f)$
 
 \begin{lemma}
 \label{product}
 
\begin{enumerate}
\item  Let $\ f$ be a.p.f. and $\ g$ such that $\lim\limits_{x\rightarrow \infty}g(x)=L$. Then $$M(gf)=L\cdot M(f)$$
 \item Let $\ f$  and $\ \lambda$ be as in the previous theorem and let  $\ g=[g(\varepsilon)]\in\overline{\mathbb{R}}$ be such that $ g\approx g_0\in\mathbb{R}$. Then $$\nu\bigg(g\cdot\bigg( f\circ\frac{1}{\lambda}\bigg)\bigg)= \lim\limits_{\eta\longrightarrow 0} \frac{1}{\eta} \int \limits_0^{\eta}g(\varepsilon) f\bigg(\frac{1}{\lambda(\epsilon)}\bigg)d\epsilon =g_0\cdot M(f)$$
 \item Let $\ f$ and $\ g$ be as in the  previous item then $$ \nu(f\circ g)= \lim\limits_{\eta\longrightarrow 0} \frac{1}{\eta} \int \limits_0^{\eta}  f\bigg(g(\varepsilon)\bigg)d\epsilon =f(g_0) $$
 \end{enumerate}
 \end{lemma}

\begin{proof} We have that $ \bigg|\frac{1}{T} \int \limits_0^Tg(x) f(x)dx-\frac{1}{T} \int \limits_0^TL f(x)dx \bigg|\leq \|g-L\|_{[0,T],\infty}\cdot\frac{1}{T} \int \limits_0^T| f(x)|dx$ \\$\leq \|g-L\|_{[0,T],\infty}\cdot (2M(|f(x)|)+1)$, for $\ T$ large,  proving the first item. The second item follows by the the previous theorem and the proof of the previous item. For the last item, note that $\ f\circ g\approx f(g_0) $ because $\ f$ is continuous. Hence the transition probability is the mean value of any sequence converging to $\ f(g_0)$.\end{proof}

\noindent From the previous lemma,  we get the following corollary.

\begin{corollary}
\label{gapf}
Let $\ f$ be an a.p.f., $\ \lambda$ a pure infinity and $ \ g\in\overline{\mathbb{R}}$ such that $\ g\approx g_0\in \mathbb{R}$. Then $\ M(g\cdot f\circ\lambda)=g_0\cdot M(f)$ and, in case $\ f\circ g$ exists,   then  $\ M(f\circ g)=f(g_0)$.
\end{corollary}
\vspace{0.cm}

The last result of this section proves that a positive function which is bounded by an a.p.f.  has generalized mean values.

\begin{theorem}
\label{dominado}

Let  $\ 0\leq |g(x)|\leq f(x)$,  $\ f$ is an a.p.f  and $\ \lambda$ a pure infinity. Then   both $\ \nu(g\circ \lambda)$ and $\ \nu(|g\circ \lambda|)$ have non-empty supports. 
\end{theorem}
 
\begin{proof}  We may suppose that $\ g$ is a positive function. For each $0<\eta <1$, sufficiently small,  we have that  $ 0\leq \frac{1}{\eta} \int \limits_0^{\eta}  g(\lambda(\varepsilon))d\epsilon\leq \frac{1}{\eta} \int \limits_0^{\eta}  f(\lambda(\varepsilon))d\epsilon \leq 2M(f)+1$. Hence, in the generalized environment, we have that $0\leq \nu(g\circ \lambda)\leq (2M(f)+1)\ $ and from this the result follows. \end{proof}

Let $\ f$ be a $\ C^{\infty}$ a.p.f. and $\nu\in \overline{\mathbb{R}}$ a positive infinitesimal.   For each $\ \varepsilon\in I$,  define the set  $\ M^{f,\nu}_{\varepsilon}=\{\tau_{\varepsilon}\in\mathbb{R}\ :\ \|f(x+\tau_{\varepsilon})-f(x)\|_{\infty}\leq \nu(\varepsilon)\}$, a set of translation numbers of $\ f$,  and consider the internal set  $\ M= [M_{\varepsilon}^{f,\nu}]$. Given $\ \tau\in M$, we have that $supp(\tau)\cap \mathbb{R}^*\neq \emptyset$ if and only if $\ f$ is a periodic function. Hence if $\ f$ is  non periodic then $\ M$ consists of  infinitesimals or infinities. If  $\ f$ is $\ C^{\infty} $, $\ \nu$ is a null and $\ M\neq \emptyset$ then the canonical image of $\ f$ in $\ \mathcal{G}({\mathbb{R}})$ is a  generalized periodic function and each non-zero element of $\ M$ is a generalized period of $\ f$. Such a periods of $\ f$ need not be an invertible element, but  if  $\ \tau\in M$ is an invertible  pure  infinity then $\ \frac{1}{\tau}\int\limits_0^{\tau}f(t)dt \approx M(f)$.  One could  start  with any moderate net of a.p. functions and not just with a constant net. A generalized function which has such nets as representatives we shall call  a {\it g.a.p.f.}, short for  {\it generalized a.p.f}.   It would be interesting to know if there are  non periodic a.p. functions which are generalized periodic. The same question can be posed for g.a.p.f. 

\section{Fock Space}

In  \cite{Colomarxiv, sponer1} Colombeau and Gsponer propose a setting  for QFT aiming to deal with the renormalization theory.  We collect some fact sabout this setting.  We will be using their notation and thus it will be useful for the reader  to survey their paper for a more precise understanding of their setting.

The Fock space, introduced by the  Soviet physicist Vladimir Aleksandrovich Fock, who did foundational work on quantum mechanics and quantum electrodynamics,   is the Hilbertian direct sum
\begin{equation} \mathbb{F}=\bigoplus_{n=0}^{+\infty}L^2_s((\mathbb{R}^3)^n),\end{equation}
where $L^2_s((\mathbb{R}^3)^0=\mathbb{C}, \  L^2_s((\mathbb{R}^3)^1=L^2(\mathbb{R}^3)$ and for $n\geq 2 \ L^2_s((\mathbb{R}^3)^n)$ is the closed subspace of the space $L^2((\mathbb{R}^3)^n)$ made of those  (complex valued) functions which are symmetric in their $n$ arguments in $\mathbb{R}^3.$ That is, an element $F$ of $\mathbb{F}$ is an infinite sequence $F=(f_0, f_1,\dots,f_n,\dots)$ such that 
\begin{equation} \|F\|^2:=|f_0|^2+\sum_{n=1}^{+\infty}\|f_n\|^2_{L^2(\mathbb{R}^{3n})}<+\infty.\end{equation}

They  introduce the generalized functions over the  Fock Space and prove that $S_\tau(t)$,  the scattering operator,  is a unitary operator in this generalized environment. According to them, this proof is actually due to Perron (see also \cite{unpublish}). We recall the following results of \cite{sponer}.

\begin{proposition}  The interacting field operator satisfies the formula.

$$\Phi(x,t)=(S_\tau(t))^{-1}\Phi_0(x,t) S_\tau(t)$$

\end{proposition}

\begin{proposition}  The scattering operator is solution of the ODE.
 
$$\partial_tS_\tau(t)=-i\frac{g}{N+1}\int_{x\in\mathbb{R}^3}(\Phi_0(x,t))^{N+1}dx \ S_\tau(t), \ \ S_\tau(\tau)=id $$

\end{proposition}

The detailed formal calculations for the proofs of propositions 4.1 and 4.2  are carried out  in \cite{Colomarxiv}.  Using this generalized setting over  the Fock Space,  Colombeau and Gsponer (see \cite{Colomarxiv, sponer}) propose how  to obtain numerical predictions of the transition probabilities, i.e.,  the probability that a state $F$ before interaction produces a state $G$ after interaction. Considering   an observable $O(x,t)$ as a product of field operators, then, after interaction,  the formula one gets is 
\begin{equation} O(x,t)=S_{\tau}(t)^{-1} O_0(x,t)S_{\tau}(t).\end{equation}

The formula $$|<F_1|O(x,t)F_2>|=|<F_1|S_\tau(t)^{-1}O_0(x,t)S_\tau(t)F_2>|= |<S_\tau(t)F_1,O_0(x,t)S_\tau(t)F_2>|$$ 

\noindent shows that $S_\tau(t) F_i$ are the states that give the results from experiments in the new free field theory after the interaction has ceased ($\tau$ in the past before interaction, $t$ in the future after interaction). Therefore the probability that a state $F$ before interaction would be transformed after interaction into a state $G$ is \ $|<G,S_\tau(t) F>|$,  which is therefore interpreted as the transition probability of $F$ into $G$ due to the interaction. A way to view  this transition probability is the following:  Let $\ u, v$ be orthogonal  vectors of norm one of a Hilbert space $\ {\mathbb{H}}$ and let $\ S$ be a unitary operator on $\ {\mathbb{H}}$. Then  one sees  $|<\vec{u},S\vec{v}>|$ as probability  in the following way:   as a probability by considering an average of the values $|<\vec{u},S(\epsilon)\vec{v}>|$ on a large number of  small values of $\epsilon$. For instance according to a classical interpretation of renormalization in the case of the anomalous magnetic moment of the electron various values of $\epsilon$ could represent various spontaneous creations  of virtual electron-positron pairs in the vacuum \cite{Penrose} section 26.9. In short,  the  transition probability should be

\begin{equation}
 \label{average}
\mathcal P_0(v\rightarrow u)=\lim_{\eta\rightarrow 0, \ N\rightarrow +\infty}\frac{1}{N}\sum_{i=1}^{N}|<\vec{u},S(\epsilon_i)\vec{v}>|, random \  \epsilon_i\in]0,\eta[
\end{equation} 

\noindent  provided this limit exists.   Note that this average is naturally related to our notion of support of a generalized object.  A non-discrete way to express this average is

\begin{equation}
 \label{average1}
	\mathcal P_0(v\rightarrow u)=\lim_{\eta\rightarrow0}\frac{1}{\eta}\int_{0}^\eta |\langle u,S_{\epsilon} v\rangle|d\epsilon. 
\end{equation}

\noindent This leads naturally to the definition of  the {\it generalized  transition probability} as

$$	\mathcal P(v\rightarrow u)=\bigg[\varepsilon\longrightarrow  \frac{1}{\varepsilon}\int_{0}^\varepsilon |\langle u,S_t v\rangle|dt\bigg]$$

\noindent Hence the existence of $\ \mathcal P_0(v\rightarrow u)$ is equivalent to 

$$\mathcal P(v\rightarrow u)\approx \mathcal P_0(v\rightarrow u)$$

\noindent or, equivalently,

$$supp(\mathcal P(v\rightarrow u))=\{\mathcal P_0(v\rightarrow u)\}$$

\noindent Since, in general,  this obviously need not be the case, this leads us to define any element of $\ supp(\mathcal P(v\rightarrow u))$ as a transition probability. If we suppose that $\ S_t$ depends continuously on $\ t$ and that  $\ supp(\mathcal P(v\rightarrow u))=\{\mathcal P_0(v\rightarrow u)\}$ then we have that

$$	\mathcal P(v\rightarrow u)\approx \bigg|  \lim_{\varepsilon\rightarrow0} \frac{1}{\varepsilon}\int_{0}^\varepsilon \langle u,S_t v\rangle dt\bigg|=\mathcal{P}_0(v\rightarrow u)$$

 Henceforth, in this paper, the existence of a transition probability will be equal to the the support of the generalized transition probability being not empty.  Colombeau and Gsponer  require that the support of the generalized transition probability has a unique element.
 
 Here we deal with the Colombeau algebra defined over $L(\mathbb{D})$ (see \cite[Section 1.5.2]{sponer1}) and thus all operators involved are, by  definition, bounded operators. This can be justified since there does  not exist a structure containing the unbounded operators in which one can calculate their norm which is  necessary in the construction of a Colombeau algebra. The construction of  $L(\mathbb{D})$ probably  contains most  important operators used  in Physics.


   \section{Transition Probabilities}
   \label{finitedim}
   
    In this section we study self adjoint operators  over $\OK^n$ and infinite dimensional Hilbert spaces.   We start looking at the finite dimensional case. In this case, a self adjoint operator corresponds to  a hermintian matrix $A\in M_n(\OK)$. Hence  we are interested in matrices with coefficients in  $\OK_{as}$  or $\tilde{\mathbb{K}}_c$  and the existence of the transition probability $\ \nu(A)= |<\vec{u},exp(iA(\epsilon))\vec{v}>|$ where $u, v$ are unitary vectors. As notice before, this need not be the case and thus we will study the support of the generalized transition probability.  Our first result gives  a condition for the existence of the transition probability.


If  $\ A=(a_{ij})$ is a matrix, we denote by  $\| A\|$ its euclidean norm which is also known as its Frobenius norm. The Frobenius norm can be extended to the generalized environment in a natural way. We shall use the same notation for this extended norm  (see \cite{AraFerJu, JJO}). We denote by $ \ S_n$ the permutation group on $\ n$ symbols. We now state a result due to Hoffman-Wielandt fundamental for what comes in the sequel. We refere the reader to \cite{bhatia, hoffman, horman1} for more information.

\begin{theorem}[Hoffman-Wielandt]
\label{hoffman}
Let $\A$ and $B$ be two normal complex matrices with spectra $\ \{\alpha_1,\dots,\alpha_n\}$ and   $\ \{\beta_1,\dots,\beta_n\}$, respectively. Then  $\min\limits_{\sigma \in S_{n}}\ \Big(\sum\limits_{1}^{n}|\alpha_i-\beta_{\sigma(i)}|^2\Big)\leq \| A-B\|^2$.
 
 \end{theorem}

Given a generalized function $\ f$ , we denote by $R_f$ the set of solutions of the equation $\ f(x)=0$ in $\ \overline{\mathbb C}$.  In case $\ f(x)$ is a generalized polinomial, we have the following result from \cite{ver2}.

\begin{lemma}
\label{ver}
Let $f(x)\in\overline{\mathbb{C}}[x]$ be monic. There exist  $\ \lambda_1,\dots,\lambda_n\in \overline{\mathbb{C}}$ such that $f(x)=(\lambda-\lambda_1)\cdots(\lambda-\lambda_n)$ and $R_f =interl(\{\lambda_1,\dots,\lambda_n\})$.

\end{lemma}
  
  If$\ A$ is a square matrix, we denote by $\ \chi_A(\lambda)$ its characteristic polinomial over $\  \overline{\mathbb{C}}$. The spectrum of $\ A$ is the set of eigenvalues of $\ A$, i.e., $\ R_{\chi_A}$. We refer the reader to \cite{horman1, ver2} for  excellent explanations with respect to the subject of eigenvalues. For our purpose, we recall that $\lambda$ is an eigenvalue of $\ A$ if and only if $\ \lambda \in Spec(A)$ (see \cite[Proposition 3.20]{horman1}). This is equivalent to the following: there exists $\ v=(v_1, \cdots v_n) \in  \overline{\mathbb{C}}^n$ such that $\ Av=\lambda v$ and $Span_{ \overline{\mathbb{C}}}[v_1,\cdots v_n]= \overline{\mathbb{C}}$. Such a vector is called a {\it free vector}.   We can now state two consequences of the previous lemma   which can be found in \cite{horman1, ver2}.
  
  \begin{theorem}
  \label{horman}
  Let $\A $ be an $n\times n$-matrix with coefficients in $\ \overline{\mathbb C}$. Then theres exists $\ \lambda_1,\dots,\lambda_n\in \overline{\mathbb C}$, eigenvalues of $\ A$ such that  $\ \chi_A(\lambda)=(\lambda-\lambda_1)\cdots(\lambda-\lambda_n)$ is the characteristic polinomial of $\ A$ and the spectrum of $\ A$ is equal to $\ interl(\{\lambda_1,\dots,\lambda_n\})$.
  
  \end{theorem}

  \begin{theorem}
  \label{horman1}
 Let $\A $ be an $n\times n$-hermitian matrix with coefficients in $\ \overline{\mathbb C}$. Then the spectrum of $\ A$ is contained in $\ \overline{\mathbb R}$. Moreover, if there exist  eigenvalues  $\ \lambda_1\geq\cdots\geq\lambda_n$  of $\ A$,  then they satisfy the conclusion of the previous  theorem and  $\ A$ is unintary equivalent to the  diagonal matrix $\ diag(\lambda_1,\dots,\lambda_n)$.
 
\end{theorem}
 
 If$\ A$ is an $n\times n$-matrix with coefficients in $\ \overline{\mathbb C}$  and     $\ \{ \lambda_1,\dots,\lambda_n\}\subset \overline{\mathbb C}$ is contained in the $\ Spec(A)$ is as in one of the previous theorems than we say that   $\ \{ \lambda_1,\dots,\lambda_n\in \overline{\mathbb C}\}$  is a {\it  generating set} for $\ Spec(A)$.  

 Let$\ A$ is an $n\times n$-matrix with coefficients in $\ \overline{\mathbb C}$  and     $\ \{ \lambda_1,\dots,\lambda_n\}\subset \overline{\mathbb C}$  a generating set.  A {\it permutation net} is a map  $\ \sigma : I\longrightarrow S_n$.  Define the element $\ \lambda_i^{\sigma}$ as the element whose representative is   $\ \lambda_i^{\sigma}(\varepsilon):= \lambda_{\sigma_{\varepsilon}(i)}(\varepsilon)$. For each $1\leq j\leq n$ let $\ S_{ij} =\{\varepsilon\in I : \sigma_{\varepsilon}(i) = j\}$ and let $\ e_{ij}$ be the idempotent generated by $\ S_{ij}$. Then  $\ \lambda_i^{\sigma}=\sum\limits_j e_{ij}\cdot \lambda_j\in Spec(A)$ and $ \{e_{i1},\cdots e_{in}\}$ is a complete set of ortogonal idempotents.  This proves that $\ \lambda_i^{\sigma}$ is a well defined element. Moreover, if the matrix $(e_{ij})$ is invertable then $\ \{\lambda_1^{\sigma},\cdots \lambda_n^{\sigma}\}$ is  a generating set for $\ Spec(A)$.

\begin{theorem}
\label{supp}
	Let  $A \in M_n({\overline{\mathbb K}}_{as})$ be a hermitian matrix   with   $supp(A)=\{ A_0\}$ and let   $ \{\lambda_1,\dots,\lambda_n\}$ be a generating set for $\ Spec(A)$.  Then there exists a permutation net $\ \sigma$  such  that $\{ \lambda_1^{\sigma}, \cdots, \lambda_n^{\sigma}\}$ is a generating set for $\ Spec(A)$, $\lambda_i^{\sigma}\in {\overline{\mathbb K}}_{as}\ \forall i\ $  and    $ \{supp(\lambda_i^{\sigma}),\cdots, supp(\lambda_n^{\sigma})\} = Spec(A_0)$.  \end{theorem}

\begin{proof} Let  $\ supp(A_0)=\{ \mu_1,\dots,\mu_n\}$ as an ordered set.  Since $A$ is hermitian and $A\approx A_0$ it follows that $\ A_0$ is hermitian and we may choose a  representative for $\ A$ such that $\ A_{\varepsilon}$ is hermitian for all $\ \varepsilon\in I$. Applying the  Hoffman-Wielandt's   Theorem, there exists a permutation net $\ \sigma$ such that

\begin{equation}
	|\lambda_i^{\sigma} -\mu_i|\leq \|A-A_0\|\approx 0 \label{eq1}
	\end{equation}
\vspace{.5cm}
\noindent from which it follows that $supp(\lambda_i^{\sigma})= \mu_i$. Since  $\lambda_i^{\sigma}  \in Spec(A)$, and $\ \prod\limits_i (\lambda - \lambda_i^{\sigma}(\varepsilon))=\prod\limits_i(\lambda - \lambda_i(\varepsilon))$ for all $\ \varepsilon$, we have that $\ \prod\limits_i (\lambda - \lambda_i^{\sigma})=\chi_A(\lambda)$ and hence   the theorem is proved.
\end{proof}

\begin{corollary}
\label{eigenvalue}
	Let  $\ A\approx A_0\in M_n({\mathbb{C}})$ and $\lambda\in Spec(A)$.   Then   $supp(\lambda)\subseteq Spec(A_0)$ with  equality  if $\lambda =\sum\limits_i e_i\cdot \lambda_i^{\sigma}$, where $\{e_1,\cdots,e_n\}$ is a complete set of non-trivial orthogonal idempotents and the $\ \lambda_i^{\sigma}$ are given by the theorem.  \end{corollary}

\begin{proof}  Since $\lambda\in Spec(A)$, there exists a free vector $v\in {\overline{\mathbb{C}}}^n$ of norm $\ 1$ such that $Av=\lambda v$. If $\lambda_0\in supp(\lambda)$, then there exists an idempotent $\ f$ such that $f\cdot\lambda \approx  f\cdot\lambda_0$. Since $\ v$ has norm one, there exists an idempotent $\ e$ such that $\ ef=e$ and $\ e\cdot v\approx e\cdot v_0$, with  $v_0\in {{\mathbb{C}}}^n$.  Using this we have, $e\cdot \lambda_0v_0\approx e\lambda v =(Av)e\approx e\cdot A_0v_0$ and hence, $A_0v_0 =\lambda_0 v_0$.

We shall use the notation of the proof of the theorem.  Let $\ supp(A_0)=\{ \mu_1,\dots,\mu_n\}$.  If $\lambda =\sum\limits_i e_i\cdot \lambda_i^{\sigma}$ with $\ e_ie_j =\delta_{ij}e_i $ and $\ e_i\notin\{0,1\}$ then $\ \lambda\cdot e_i=e_i\cdot \lambda_i^{\sigma}\approx  e_i\cdot \mu_i$ and hence the result follows. \end{proof}

 \begin{corollary}
	$supp(Spec(A)):=\bigcup\limits_{\lambda\in Spec(A)}supp(\lambda)=Spec(A_0).$
\end{corollary}

\begin{remark}
If $\ A$  is hermitian matrix   and  $B$ is an arbitrary normal matrix and we order $\ Spec(A)$ and  $\ Spec(B)$ by the real part of its elements, i.e.,  $\alpha_1\geq\cdots\geq\alpha_n$ and  $Re(\beta_1)\geq\cdots\geq Re(\beta_n)$ then we can give a more precise description of the $\ \lambda_i^{\sigma}$'s. We refere the reador to  \cite[Observation 15.3]{bhatia} (see also \cite[Lemma 3.24]{horman1}.  In fact, in this case we have that, for each $\ i$,  $\lambda_i{\sigma}=\lambda_i$.  
\end{remark}

\begin{proposition} 
\label{probabil}
Let $\ A$ and $\ A_0$ be as in the previous theorem and let $u_0,v_0$ be unitary orthogonal vectors of  $\ \mathbb C^n$. Then  $supp(\nu(A))=\{\nu(A_0)\}$. In particular, $\nu(A)\approx \nu(A_0)$. 
\end{proposition}

\begin{proof} We first note that if $\ e$ is and idempotent than  $\ exp(iA)\cdot e = exp(i(eA))\approx exp(iA_0)$. Now let $t_0\in supp(\nu(A))$ and let $\ e$ be an idempotent such that $e\cdot \nu(A)\approx t_0$.  Using this,  one gets that $t_0\approx e\cdot  \nu(A)  =exp(i(eA))\approx exp(iA_0)=\nu(A_0)$ and hence $\ t_0=\nu(A_0)$.    
\end{proof}

For each   $1\leq i\leq n$, denote  by $V_{\mu_i}(A_0)$  the eigenspace of $A_0$ corresponding to the eigenvalue $\mu_i\in Spec(A_0)$  and let $\ V_{\lambda_i}(A)$ be the  $\overline{\mathbb K}-$sub-modulo de $\overline{\mathbb K}^n$ spanned by the eigenvectors of  $A$  corresponding to the eigenvalue $\ \lambda_i$.

Suppose that $A$, $\ A_0$, $\sigma$ and $\{\lambda_1^{\sigma},\cdots,\lambda_n^{\sigma}\}$  are as in Theorem~\ref{supp}. Let $\ Spec(A_0)=\{\mu_1,\cdots,\mu_n\}$.  By \cite[Propositions 3.18 and 3.25]{horman1}, $ Spec(A)\subset {\overline{\mathbb{R}}}$ and $\ A$ is unitary equivalent to $\ diag(\lambda_1^{\sigma},\cdots,\lambda_n^{\sigma})$. It also follows easily that   $\ supp(\chi_A)=\{\chi_{A_0}\}$ and the algebraic multiplicity of $\  \lambda_i^{\sigma}$ and $\ \mu_i$ are equal.

\begin{corollary}
	For each  $1\leq i\leq n$, we have  $supp(V_{ \lambda_i^{\sigma}}):= \displaystyle\bigcup_{v\in V_{\lambda_i^{\sigma}}(A)}supp(v)= V_{\mu_i}(A_0)$.
\end{corollary}

\begin{proof} Let  $v\in V_{\lambda_i^{\sigma}}(A)$ and let  $v_0\in supp(v)$.  Then $\ e\cdot v\approx e\cdot v_0$, for some $ \ e^2 =e$. With the notation of the  proof of the theorem, it follows that $e\cdot(A_0v_0)\approx e\cdot(Av)=e\cdot (\lambda_i^{\sigma}v)\approx e\cdot \mu_i v_0$. Consequently, $v_0 \in V_{\mu_i}$. Since the algebraic multiplicity of $\  \lambda_i^{\sigma}$ and $\ \mu_i$ are equal, it follows that the ${\overline{\mathbb{C}}}-$free rank of $\ V_{ \lambda_i^{\sigma}}$ is equal to the $\ {\mathbb{R}}-$dimension of $V_{\mu_i}$ and hence, since $supp(V_{ \lambda_i^{\sigma}})$ is an $\ {\mathbb{R}}-$vector space, we have that  $supp(V_{ \lambda_i^{\sigma}}) = V_{\mu_i}(A_0)$. \end{proof}

Note that if  $A\in M_n(\tilde{\mathbb K}_c)$ then we have that  $supp(A)\neq\emptyset$. Moreover, its support is compact.  We are ready to prove our main result for transition probabilities in the finite dimensional case.

\begin{theorem}
\label{proba-geral}

	Let  $A=[(A_\epsilon)_\epsilon]\in M_n(\tilde{\mathbb K}_c)$ be a hermitian matrix. Then $\ supp(\nu(A))=\{\nu(A_0): A_0\in supp(A)\}$.    
\end{theorem}

\begin{proof}  If $A_0\in supp(A)$ then there exists an idempotent $\ e$ such that $e\cdot A\approx e\cdot A_0$. Applying Proposition~\ref{probabil} we have that $e\cdot\nu(A) =\nu(eA)\approx \nu(eA_0)=e\cdot\nu(A_0)$.  Hence $\ \nu(A_0)\in supp(\nu(A))$. Conversely, let $\ a\in supp(\nu(A))$ and $\ e=e^2$ such that $e\cdot a\approx e\cdot\nu(A)=\nu(e\cdot A)$. Choose $\ A_0\in supp(e\cdot A)\subset supp(A)$ and $\ f=f^2$ such that $\ ef=f$ and $f\cdot A_0\approx f\cdot(eA) =f\cdot A$. Then $\ f\cdot a \approx f\cdot\nu(A)=\nu(f\cdot A)\approx \nu(f\cdot A_0)=f\cdot\nu(A_0)$. Hence $\ a=\nu(A_0)$ and the theorem is proved. \end{proof}



 We now look at  $\overline{\mathbb{K}}$-Hilbert modules, $\ {\mathcal{G}}_H$ where  $\ H$  is a infinite dimensional Hilbert space.  We refere the reader to \cite{garetto, garetto1} for terminology and results for these modules. We start formalizing the notion of support in this environment. We shall also need a version of the Hoffman-Wielandt  Theorem for Hilbert spaces for which we refer the reader to \cite{bhatiainf}.

\begin{definition}
	Let  $T=(T_\epsilon): {\mathcal{G}}_H\rightarrow {\mathcal{G}}_H$ be a basic map (see \cite{garetto, garetto1}). Its support, denoted  by   $supp(T)$,is defined as   $supp(T):=\{T_0\in {\mathcal{B}}(H):\exists \ e=e^2$, such that $\ e\cdot T \approx e\cdot T_0 , i.e.,  \|e\cdot (T-T_0)\|_{\infty} \rightarrow 0\}$.
\end{definition}

 \begin{theorem}
 \label{hoffinf}
	Let $\ T$ and  $S$  be normal  Hilbert-Schmidt operators with respective sets of eigenvalues   $\{\alpha_1,\alpha_2,\dots\},\{\beta_1,\beta_2,\dots\}$. Then, for each   $\varepsilon>0$ there exists a permutation  $\pi\in Sym({\mathbb{N}})$  such that 
	\begin{equation}
	\label{hofwi}
		\Bigg[\sum_{i=1}^\infty |\alpha_i-\beta_{\pi(i)}|^2\Bigg]^{1/2}\leq||T-S||+\varepsilon
	\end{equation}       
	
\end{theorem}

The interested reader can find the prove of this theorem in   \cite{bhatiainf}. Note that in Equation~\ref{hofwi} we can change $\ \varepsilon$ by a null function $\ \xi(\varepsilon)$.  In what follows the norm involved is the Hilbert-Schmidt norm which we extend in a natural way to the generalized environment. Let $\pi = (\pi_{\varepsilon})\in [Sym({\mathbb{N}})]^{]0,1]}$ be a permutation net,  $\ T=(T_\epsilon):{\mathcal{G}}_H\rightarrow {\mathcal{G}}_H$ a basic map and $\ \alpha =[\varepsilon\longrightarrow \varepsilon]$ be the standard gauge.  For each $\ \varepsilon$ let $\ (t_n(\varepsilon))_{n\in {\mathbb{N}}}$ be an enumeration of the eigenvalues of $\ T_{\varepsilon}$ and write $\ t(\alpha) := (t_n(\alpha))$ for all these enumerations. Define $\ t^{\pi}(\alpha) : = (t_{\pi_{\varepsilon}(n)}(\alpha))$ and write the latter as  $\  (t_{\pi_{\alpha}(n)}(\alpha))$. Since  every $\ T_{\varepsilon}$ is Hibert-Schmidt, we have that  $\ (t_n(\varepsilon))_{n\in {\mathbb{N}}}\in l_2({\mathbb{R}})$. If we choose the   $\ T_{\varepsilon}$ 's in a ball then we have that $t(\alpha)\in {\mathcal{G}}(l_2({\mathbb{R}}))$. Proceeding in the same way as in the finite dimensional case we obtain that  $ t^{\pi}(\alpha)  = \sum\limits_{j=1}^{\infty} e_{nj}\cdot t_j(\alpha)$, where $\ e_{nj}$  is the charateristic function of the set $\ S_{nj}:=\{\varepsilon \ :\ \pi_{\varepsilon}(n)=j\}$. We can now state the former theorem in the setting of $\ \overline{\mathbb{K}}$-Hilbert modules.

 \begin{corollary}
 \label{hoffinf1}
	 
	 Let  $T=[(T_\epsilon)]:{\mathcal{G}}_H\rightarrow {\mathcal{G}}_H$. Suppose that $\ T$ is a self-adjoint Hilbert-Schmidt operator, i.e. $T=(T_{\varepsilon})$ with each $ \  T_{\varepsilon}$ a selfadjoint  Hilbert-Schmidt operator, $\ T_0\in supp(T)$ and $\ e=e^2$ is such that $\ e\cdot T\approx e\cdot T_0$. If  $\ t_0 = \{t_1,t _2,\dots\}$  is an enumeration of the eigenvalues of $\ T_0$, then there exists a permutation  net $\pi$  such that 
	\begin{equation}
		e\cdot\Bigg[\sum_{n=1}^\infty |t_n-t_{\pi_{\alpha}(n)}(\alpha)|^2\Bigg]^{1/2}\leq||e\cdot(T-T_0)||+e\cdot\alpha\approx 0
	\end{equation}       
	
	\vspace{0.5cm}
	\noindent Moreover $\ t(\alpha) \in l_2({\overline{\mathbb{R}}})$ and  $t^{\pi}_n \longrightarrow t_0$.
\end{corollary}

Denote  by $\ {\mathcal{B}}({\mathcal{G}}_H)$  the elements $T=(T_\epsilon): {\mathcal{G}}_H\rightarrow {\mathcal{G}}_H$  for  which there exists a representative  $(T_\epsilon)_\epsilon$ of  $\ T$ and a compact subset   $K\subset {\mathcal{B}}(H)$ such that  $T_\epsilon\in K$  for $\  \epsilon<\eta$, for some   $\eta\in I$. We call such a $\ T$ a {\it compactly supported element}.  Note that if  $\ T\in {\mathcal{B}}({\mathcal{G}}_H)$ then $\ supp(T)\neq\emptyset$.  Note that \cite[Lemma 10.1]{V3} applies in this case.

\begin{theorem}
\label{bounded}
Let  $T= [(T_\epsilon)]\in \ {\mathcal{B}}({\mathcal{G}}_H)$   be a  selfadjoint   operator   such that each $T_\epsilon$ is selfadjoint. Then  $supp(\nu(T))=\{\nu(T_0):T_0\in supp(T)\}$.  

\end{theorem}

\begin{proof} Let  $T_0\in supp(T)$ and $\ e =e^2$ such that $\ e\cdot T\approx e\cdot T_0$.   Since all operators involved are bounded  and the exponencial map is continuous, we have that   $\ e\cdot\nu(T)= e\cdot\nu(e\cdot T)\approx e\cdot\nu(e\cdot T_0)=e\cdot \nu(T_0)$.  Conversely, let  $\ \omega\in supp(\nu(T))$  and $ \ e =e^2$ such that $\ e\cdot\omega\approx e\cdot \nu(T)$. Choose $\ f$ such that $\ f\cdot e =f$ and $\ f\cdot T \approx T_0$, for  some $\ T_0$. Then we have that $\ f\cdot\omega\approx f \cdot \nu(T)=\nu(f\cdot T)\approx \nu(f\cdot T_0) =f\cdot\nu(T_0)$. It follows that $\ \omega =\nu(T_0)$, completing the  proof. \end{proof}


\section{Examples}

In this section we  considere examples some of which   were already considered in \cite{unpublish} without however explicit calculations.  These examples are not covered by the results of the previous sections. Before we continue we make the following definition. An element $\lambda \in {\overline{\mathbb{R}}}$ which satisfies $supp(\lambda)=\emptyset$ we shall call a {\it pure infinity}. Note that this is equivalent to $\ \lim\limits_{\varepsilon \longrightarrow \infty}\lambda (\varepsilon)=\infty$.

 In a 2-dimensional Hilbert space let   $H(t)=\left(\begin{array}{ll}
      \frac{g}{\epsilon}&0 \\ 0 & \frac{g}{\epsilon^2}
            \end{array} \right)$ $=\left(\begin{array}{ll}
      g\alpha^{-1}&0 \\ 0 & g\alpha^{-2}
            \end{array} \right)$,  where $\ g\in{\mathbb{R}}$. The solution of the ODE  $$S^{\prime}(t)=iH(t) S(t)$$
            
            \noindent   is

            $$\Bigg[\varepsilon \longrightarrow S_{\varepsilon}(t)= \left(\begin{array}{ll}
      \exp\bigg({i\frac{g}{\epsilon}(t-t_0)}\bigg)& \ \ \ \ \ \ \ 0 \\ 0&\exp\bigg({i\frac{g}{\epsilon^2}(t-t_0)\bigg)}
            \end{array} \right)\Bigg]$$

            \noindent Hence, for orthogonal unitary vectors $\ u=(u_1, u_2)$ and $\ v=(v_1, v_2)$ we have   that 
            
$$\langle u\mid S_\epsilon(t)v\rangle=u_1 v_1\exp\bigg({i\frac{g}{\epsilon}(t-t_0)}\bigg)+u_2v_2\exp\bigg({i\frac{g}{\epsilon^2}(t-t_0)}\bigg)$$

\noindent In \cite{unpublish} is it asserted that the support of $\langle u,S_\epsilon(t)v\rangle $ is not unitary, i.e., $ \langle u,S_\epsilon(t)v\rangle$ is not associated to any real number. But clearly  $[\varepsilon \longrightarrow \langle u,S_\epsilon(t)v\rangle] $ is the sum of the composition of   a.p. functions and pure infinities. Hence it does have a mean value.  Hence it is most likely that $|\langle u,S_\epsilon(t)v\rangle |$ also does have a mean value. We proceed to compute this value. If write $\ a=u_1v_1,  b=u_2v_2$  then $\ a+b=0$ and we have that 

$$  | \langle u,S_\epsilon(t)v\rangle| =2|a|\cdot\bigg |sin\bigg(\frac{1-\varepsilon}{\varepsilon^2}g(t-t_0)\bigg)\bigg|$$

\noindent It follows that $\ M( | \langle u,S_\epsilon(t)v\rangle |) =2|a|\cdot M\Bigg(\bigg |sin\bigg(\frac{1-\varepsilon}{2\varepsilon^2}(t-t_0)\bigg)\bigg|\Bigg)=\frac{4}{\pi}\cdot |a|$, where we used Theorem~\ref{meanval1} in the last equality.  We refer the reader to  \cite[Section 8]{Colomarxiv} (see also \cite[Theorem 5]{unpublish}).

We look at another  example from \cite{ unpublish, transition},  considering  the  $2\times 2$ symmetric matrix\\ $A = \Bigg[\varepsilon\longrightarrow\left(\begin{array}{ll}
      \frac{a}{\epsilon^{n_1}+\epsilon^{2n_1}}& \frac{c}{\epsilon-2\epsilon^2}\\ \frac{c}{\epsilon-2\epsilon^2}&\ \frac{b}{\epsilon^{n_2}}
            \end{array} \right)\Bigg]$ $= \left(\begin{array}{ll} \frac{a}{\alpha^{n_1}+\alpha^{2n_1}}& \frac{c}{\alpha-2\alpha^2}\\ \frac{c}{\alpha-2\alpha^2}&\ \frac{b}{\alpha^{n_2}}
            \end{array} \right)$.  If  $\ u=(u_1, u_2), \ v= (v_1, v_2) $ and one computes

$$ \frac{1}{N}\sum_{j=1}^N|\langle (u_1, u_2)\mid  \exp(-iH)(v_1,  v_2)\rangle|$$

\noindent by choosing $\epsilon$ at random between  0 and some very small value $\eta$; $<\dots>$ is the standard Hermitian scalar product in $\mathbb{C}^n$   and uses the matlab routine \textit{expm} to calculate the exponential of a matrix its seems that a  transition probability should exist, for $\ a=1, b=1.3,  c=0.4, n_1=n_2=1$ and is around $0.9251$. This simulation is reported in \cite{unpublish, transition}. Note however that this simulation only indicates that the support of the generalized transition probability is not empty. However, since values of $\ \varepsilon$ were randomly chosen, it does suggest that the support has a unique element and hence the existence of the  transition probability. 

 We first consider $\  u=\bigg(\frac{\sqrt{2}}{2}, \frac{\sqrt{2}}{2}\bigg), v= \bigg(\frac{\sqrt{2}}{2}, -\frac{\sqrt{2}}{2}\bigg)$,  $\ A =\Bigg[\varepsilon\longrightarrow \left(\begin{array}{ll}
      \frac{1}{\epsilon+\epsilon^{2}}& \frac{1}{\epsilon-2\epsilon^2}\\ \frac{1}{\epsilon-2\epsilon^2}&\ \ \ \frac{1}{\epsilon}
            \end{array} \right)\Bigg]$ and prove the existence of the transition probability for any orthogonal pair and give its actual value.    We start calculating   the eigenvalues of $\ A$: 
 
 $$\lambda(\varepsilon) =  \frac{1}{2}\Bigg[ \frac{2\varepsilon+\varepsilon^2}{\varepsilon^2+\varepsilon^3} \pm \sqrt{\bigg(\frac{2\varepsilon+\varepsilon^2}{\varepsilon^2+\varepsilon^3}\bigg)^2 - 4\bigg[\frac{1}{\varepsilon^2+\varepsilon^3} -\frac{1}{(\varepsilon-2\varepsilon^2)^2}  \bigg]}\Bigg]$$
 
 \noindent From this it follows that

 $$\alpha \lambda  = \Bigg[\varepsilon\longrightarrow  \frac{1}{2}\Bigg[ \frac{2+\varepsilon }{1+\varepsilon } \pm \sqrt{\bigg(\frac{2+\varepsilon }{1+\varepsilon }\bigg)^2 - 4\bigg[\frac{1}{1+\varepsilon } -\frac{1}{(1-2\varepsilon)^2}  \bigg]}\Bigg]\Bigg]$$
 
 \noindent and hence $\alpha\cdot \lambda \approx \frac{1}{2}[2\pm 2]$.  We denote by $\ \lambda_1$ the eigenvalue with the $\ +$ sign chosen and by $\ \lambda_2$ the one with the $\ -$ sign  chosen.  From this and the use of a graphic calculator one proves that $\|\lambda_1\|=e$ and $\|\lambda_2\|<e^{-18}<1$. Hence $\ \lambda_1$ is a pure infinity and $\  \lambda_2$ is an infinitesimal.  We also have that 
 
 $$\alpha(\lambda_1-\lambda_2) = \sqrt{\bigg(\frac{2 +\varepsilon}{1+\varepsilon}\bigg)^2 - 4\bigg[\frac{1}{1+\varepsilon} -\frac{1}{(1-2\varepsilon)^2}  \bigg]} \approx \sqrt{2}$$
 
  \noindent proving that $\ \lambda_1-\lambda_2$ is a pure infinity.  A basis of eigenvectors  of  $\ A$ is  $\{(1,k), (-k, 1)\}$, where 
 
 $$\ k(\varepsilon)=   (\varepsilon-2\varepsilon^2)(\lambda_1-(\varepsilon+\varepsilon^2)^{-1})=$$  $$ (\varepsilon-2\varepsilon^2) \frac{1}{2}\Bigg[ \frac{1}{\varepsilon}  +\sqrt{\bigg(\frac{2\varepsilon+\varepsilon^2}{\varepsilon^2+\varepsilon^3}\bigg)^2 - 4\bigg[\frac{1}{\varepsilon^2+\varepsilon^3} -\frac{1}{(\varepsilon-2\varepsilon^2)^2}  \bigg]}\Bigg]=$$   $$ \frac{1}{2}\Bigg[1+ \sqrt{\bigg(\frac{2+\varepsilon }{1+\varepsilon }\bigg)^2 - 4\bigg[\frac{1}{1+\varepsilon } -\frac{1}{(1-2\varepsilon)^2}  \bigg]}\Bigg] \approx 1.5:=k_0 $$
 
 \noindent Hence $\  M =\frac{1}{\sqrt{1+k^2}} \left(\begin{array}{ll}
     1& -k\\ k&1    \end{array} \right)$ conjugates  $\ A$ into diagonal form $\ D=  \left(\begin{array}{ll}
     e^{i\lambda_1}& 0\\ 0&e^{i\lambda_2}    \end{array} \right)$ and $ \ M\approx M_0 =  \frac{1}{\sqrt{1+k_0^2}} \left(\begin{array}{ll}
     1& -k_0\\ k_0&1    \end{array} \right)$. From this it follows that 
     
     $$\langle u\mid exp(iA)v\rangle =\langle Mu\mid DMv\rangle=: g_1g_3e^{i\lambda_1}+g_2g_4e^{i\lambda_2}$$

     \noindent where $\ Mu=(g_1,g_2), Mv=(g_3,g_4)$,  $g_i\approx y_i\in \mathbb{R}, \ i=1,4$, $g_1g_3+g_2g_4=0$ and $\ \{e^{i\lambda_1}, e^{i\lambda_2}\}$ consists of pure vibrations.   We already saw that   $\ \lambda_1$ is a pure infinity.       As in the previous example, one gets that 
     
     $$\ |\langle u\mid exp(iA)v\rangle| = 2|g_1g_3|\cdot \bigg|sin\bigg(\frac{\lambda_1-\lambda_2}{2}\bigg) \bigg|$$
     
     \noindent and hence  the transition probability is $$\frac{4 }{\pi}\cdot |g_1(0)g_3(0)|=\frac{4 }{\pi}\cdot \bigg|\frac{1}{\sqrt{10}}\bigg(\frac{\sqrt{2}}{2}-k_0\frac{\sqrt{2}}{2}\bigg)\bigg(\frac{\sqrt{2}}{2}+k_0\frac{\sqrt{2}}{2}\bigg)\bigg|=\frac{2(k_0^2-1)}{(k_0^2+1)\pi}  \approx 0.244853758603$$

Clearly   these two examples show how to solve the generic case of a $\ 2\times 2$ generalized  matrix. We now  compute the actual value reported in   \cite{unpublish, transition}.  In this case,  $A =\Bigg[\varepsilon\longrightarrow \left(\begin{array}{ll}\frac{1}{\epsilon+\epsilon^{2}}& \frac{0.4}{\epsilon-2\epsilon^2}\\ \frac{0.4}{\epsilon-2\epsilon^2}&\ \ \frac{1.3}{\epsilon^{2}}
            \end{array} \right)\Bigg]$ with eigenvalues

            $$\lambda(\varepsilon) =  \frac{1}{2}\Bigg[ \frac{1.3+2.3\varepsilon}{\varepsilon^2+\varepsilon^3} \pm \sqrt{\bigg(\frac{1.3+2.3\varepsilon}{\varepsilon^2+\varepsilon^3}\bigg)^2 - 4\bigg[\frac{1.3}{\varepsilon^3+\varepsilon^4} -\frac{0.16}{(\varepsilon-2\varepsilon^2)^2}  \bigg]}\Bigg]$$

 \noindent From this it follows that

 $$\alpha^{2.6}\cdot \lambda(\varepsilon) =  \frac{1}{2}\Bigg[ \varepsilon^{.06}\cdot\frac{1.3+2.3\varepsilon }{1+\varepsilon } \pm \sqrt{\varepsilon^{1.2}\bigg(\frac{1.3+2.3\varepsilon }{1+\varepsilon }\bigg)^2 - 4\varepsilon^{0.2}\bigg[\frac{1.3(1-2\varepsilon)2-0.16(\varepsilon+\varepsilon^2)}{(1+\varepsilon)(1-2\varepsilon)}\bigg]}  \Bigg]$$
 
 \noindent and hence $\alpha^{2.6}\cdot \lambda \approx 0$.  We denote by $\ \lambda_1$ the eigenvalue with the $\ +$ sign chosen and by $\ \lambda_2$ the one with the $\ -$ sign  chosen.  Then  $\alpha^{2.5}\cdot \lambda_1 \not\approx 0$. And hence $\  \|\lambda_1\|\geq e^{2.5}$ and $ \ \lambda_1$ is a pure infinity. For $\ \lambda_2$, using a graphic calculator,  we find that $\ \alpha^{-18.2}\cdot \lambda_2\approx 0$ but $\ \alpha^{-18.2}\cdot \lambda_2\not\approx 0$. It follows that $\alpha_2$ is an infinitesimal of norm less the $\ e^{-18.2}$.

 $$\alpha^2( \lambda_1-\lambda_2)  = \Bigg[\varepsilon \longrightarrow  \sqrt{\bigg(\frac{1.3+2.3\varepsilon}{1+\varepsilon}\bigg)^2 - 4\bigg[\frac{1.3\varepsilon}{1+\varepsilon} -\frac{0.16\varepsilon^2}{(1-2\varepsilon)^2}  \bigg]\ } \Bigg]  \approx 1$$

\noindent proving that $\ \lambda_1-\lambda_2$ is a pure infinity.  A basis of eigenvectors  of  $\ A$ is  $\{(1,k), (-k, 1)\}$, where 
 
 $$\ k(\varepsilon)=   \frac{5}{2}(\varepsilon-2\varepsilon^2)(\lambda_1-(\varepsilon+\varepsilon^2)^{-1})=$$
 $$  \frac{5}{2}(\varepsilon-2\varepsilon^2)\cdot \frac{1}{2}\Bigg[ \frac{1.3}{\varepsilon^2}+ \sqrt{\bigg(\frac{1.3+2.3\varepsilon}{\varepsilon^2+\varepsilon^3}\bigg)^2 - 4\bigg[\frac{1.3}{\varepsilon^3+\varepsilon^4} -\frac{0.16}{(\varepsilon-2\varepsilon)^2}  \bigg]}\Bigg]$$

 $$  \frac{5}{4}(1-2\varepsilon)     \Bigg[ \frac{1.3}{\varepsilon}+ \sqrt{\bigg(\frac{1.3+2.3\varepsilon}{\varepsilon+\varepsilon^2}\bigg)^2 - 4\bigg[\frac{1.3}{\varepsilon+\varepsilon^2} -\frac{0.16}{(1-2\varepsilon)^2}  \bigg]}\Bigg]$$

 \noindent It follows that $\alpha k\approx 2.6$ and thus $\ k$   is a pure infinity.  Consequently,  $\  M =\frac{1}{\sqrt{1+k^2}} \left(\begin{array}{ll}
     1& -k\\ k&1    \end{array} \right)$ conjugates  $\ A$ into diagonal form $\ D=  \left(\begin{array}{ll}
     e^{i\lambda_1}& 0\\ 0&e^{i\lambda_2}    \end{array} \right)$ and $ \ M\approx M_0 =  \left(\begin{array}{ll}
     0& -1\\ 1& 0    \end{array} \right)$. From this it follows that 
     
     $$\langle u\mid exp(iA)v\rangle =\langle Mu\mid DMv\rangle=: g_1g_3e^{i\lambda_1}+g_2g_4e^{i\lambda_2}$$

     \noindent where $\ Mu=(g_1,g_2), Mv=(g_3,g_4)$,  $g_i\approx y_i\in \mathbb{R}, \ i=1,4$, $g_1g_3+g_2g_4=0$ and $\ \{e^{i\lambda_1}, e^{i\lambda_2}\}$ consists of pure vibrations. If follows from Corollary~\ref{gapf} that $\langle u\mid exp(iA)v\rangle$ has a mean value. Once gain   $\ \lambda_1$ is a pure infinity and, consequently, 
     
     $$\ |\langle u\mid exp(iA)v\rangle| = 2|g_1g_3|\cdot \bigg|sin\bigg(\frac{\lambda_1-\lambda_2}{2}\bigg)\bigg| $$
     
     \noindent   concluding  that the transition probability is $$\frac{4 }{\pi}\cdot |g_1(0)g_3(0)|=\frac{4 }{\pi} \bigg|\bigg(-\frac{\sqrt{2}}{2}\bigg)\bigg( \frac{\sqrt{2}}{2}\bigg)\bigg|=\frac{2}{\pi}\approx    0.636619772368$$

     This value is different from the one found in \cite{unpublish, transition}. From our calculations, in the dimension $\ 2$ case, it follows that if the support of the generalized transition probability has more dan one element  then the same will be  true for the element $\ |g_1\cdot g_3|$.

         Consider    $A = \Bigg[\varepsilon\longrightarrow \left(\begin{array}{ll}
      \frac{1}{\sqrt{\epsilon}+\epsilon^{2 }}& \frac{2}{\epsilon-2\epsilon^2}\\ \frac{2}{\epsilon-2\epsilon^2}&\ \ \frac{1}{\epsilon}
            \end{array} \right)\Bigg]$ and let us  calculate the mean value $|\langle u\mid \exp(-iA)v\rangle|$.
     
    The eigenvalues of $\ -A$ are
 
 $$\lambda(\varepsilon) =  \frac{1}{2}\Bigg[ \frac{1+\sqrt{\varepsilon}+\varepsilon\sqrt{\varepsilon}}{\varepsilon(1+\varepsilon\sqrt{\varepsilon})} \pm \sqrt{\bigg( \frac{1+\sqrt{\varepsilon}+\varepsilon\sqrt{\varepsilon}}{\varepsilon(1+\varepsilon\sqrt{\varepsilon})}\bigg)^2 - 4\bigg[\frac{\sqrt{\varepsilon}(1-2\varepsilon)^2-\varepsilon\sqrt{\varepsilon}-1}{\varepsilon^2(1+\varepsilon\sqrt{\varepsilon})(1-2\varepsilon)^2}\bigg]}\  \Bigg]$$
 
 \noindent From we get that

  $$\alpha\lambda = \Bigg[\varepsilon\longrightarrow \frac{1}{2}\Bigg[ \frac{1+\sqrt{\varepsilon}+\varepsilon\sqrt{\varepsilon}}{1+\varepsilon\sqrt{\varepsilon}} \pm \sqrt{\bigg( \frac{1+\sqrt{\varepsilon}+\varepsilon\sqrt{\varepsilon}}{1+\varepsilon\sqrt{\varepsilon}}\bigg)^2 - 4\bigg[\frac{\sqrt{\varepsilon}(1-2\varepsilon)^2-\varepsilon\sqrt{\varepsilon}-1}{(1+\varepsilon\sqrt{\varepsilon})(1-2\varepsilon)^2}\bigg]} \ \Bigg]\Bigg]$$
 
 \noindent and hence $\alpha\cdot \lambda \approx \frac{1}{2}[1\pm \sqrt{5}]$  We denote by $\ \lambda_1$ the eigenvalue with the $\ +$ sign chosen and by $\ \lambda_2$ the one with the $\ -$ sign  chosen.  From this it follows  that $\ \lambda_1$ and $\ \lambda_2$  are   pure infinities. Also 
 
 $$\  \alpha(\lambda_1-\lambda_2)\approx \sqrt{5}$$

 \noindent  and hence $\lambda_1-\lambda_2$  is an infinity whose norm is $\ \geq e^{-1}$.   A basis of eigenvectors  of  $\ A$ is  $\{(1,k), (k, -1)\}$, where 
 
 $$\ k(\varepsilon)=   \frac{(\varepsilon-2\varepsilon^2)}{2}\bigg(\lambda_1- \frac{1}{\sqrt{\epsilon}+\epsilon^{2 }}\bigg)=$$ $$ \frac{1-2\varepsilon}{2}\Bigg[ 1 + \sqrt{\bigg( \frac{1+\sqrt{\varepsilon}+\varepsilon\sqrt{\varepsilon}}{1+\varepsilon\sqrt{\varepsilon}}\bigg)^2 - 4\bigg[\frac{\sqrt{\varepsilon}(1-2\varepsilon)^2-\varepsilon\sqrt{\varepsilon}-1}{(1+\varepsilon\sqrt{\varepsilon})(1-2\varepsilon)^2}\bigg]} \ \Bigg] \approx \frac{1+\sqrt{5}}{2}=:k_0$$

 \noindent Hence $\  M =\frac{1}{\sqrt{1+k^2}} \left(\begin{array}{ll}
     1& -k\\ k&1    \end{array} \right)$ conjugates  $\ A$ into diagonal form $\ D=  \left(\begin{array}{ll}
     e^{i\lambda_1}& 0\\ 0&e^{i\lambda_2}    \end{array} \right)$ and $ \ M\approx M_0 =  \frac{1}{\sqrt{1+k^2_0}} \left(\begin{array}{ll}
     1& -k_0\\ k_0&1    \end{array} \right)$. From this it follows that 
     
     $$\langle u\mid exp(iA)v\rangle =\langle Mu\mid DMv\rangle=: g_1g_3e^{i\lambda_1}+g_2g_4e^{i\lambda_2}$$

     \noindent where $\ Mu=(g_1,g_2), Mv=(g_3,g_4)$,  $g_i\approx y_i\in \mathbb{R}, \ i=1,4$, $g_1g_3+g_2g_4=0$ and $\ \{e^{i\lambda_1}, e^{i\lambda_2}\}$ consists of pure vibrations. If follows from Corollary~\ref{gapf} that $\langle u\mid exp(iA)v\rangle$ has a mean value.  Note  that $\ \lambda_1-\lambda_2$ is a pure infinity.  As in the previous example, one gets that 
     
     $$\ |\langle u\mid exp(iA)v\rangle| = 2|g_1g_3|\cdot \bigg|sin\bigg(\frac{\lambda_1-\lambda_2}{2}\bigg) \bigg|$$
     
     \noindent As before, we conclude that the transition probability is $$\frac{4 }{\pi}\cdot |g_1(0)g_3(0)|=\frac{4}{\pi}\cdot \bigg|\frac{1}{1+k_0^2}\bigg(\frac{\sqrt{2}}{2}+k_0\frac{\sqrt{2}}{2}\bigg)\bigg(\frac{\sqrt{2}}{2}-k_0\frac{\sqrt{2}}{2}\bigg)\bigg|=\frac{2(k_0^2-1)}{\pi(k_0^2+1)} \approx 0.284705017367$$

\noindent The reported value in \cite{unpublish, transition} is $\ 0.1539$.


We now consider the general finite dimensional case. Proceeding in the same way for an $\ n\times n$ symmetric  matrix $\ A$, with  $\ Spec(A)=\{\lambda_1,\cdots,\lambda_n\}$,  one gets that 

\begin{equation}
\label{phasor}  |\langle u\mid exp(iA)v\rangle| =\sqrt{\|a\|^2 +2\sum\limits_{i< j}a_ia_jcos(\lambda_i-\lambda_j)} $$
$$  |\langle u\mid exp(iA)v\rangle| =  2\sqrt{-\sum\limits_{i < j}a_ia_jsin^2\bigg(\frac{\lambda_i-\lambda_j}{2}\bigg)} $$
$$  |\langle u\mid exp(iA)v\rangle|\leq  2 \bigg(\sum\limits_{i < j}\sqrt{|a_i|\cdot |a_j|\ }\cdot\bigg| sin\bigg(\frac{\lambda_i-\lambda_j}{2}\bigg)\bigg|\bigg)$$
$$ a=(a_1,\cdots, a_n)\in ({\overline{\mathbb{C}}_{as}})^n\ \mbox{and}\ \sum\limits_i a_i=0, \|a\|^2 <1\end{equation}

\noindent If the set $\ \{\lambda_i-\lambda_j \ :\  \lambda_i \neq \lambda_j  \in Spec(A)\}$ consists of pure infinities and infinitesimals and   $\ supp(a)=\{a_0\}\in {\mathbb{R}^n}$ then  $\ |\langle u\mid exp(iA)v\rangle|^2 $ has a mean value (similar to the case $\ n=2$).   One can use  Phasor addition to reduz this sum and be able to see if $\ |\langle u\mid exp(iA)v\rangle|$ has a mean value. Probably $\ \{\lambda_i-\lambda_j \ :\  \lambda_i ,\lambda_j  \in Spec(A)\}$ consisting  of pure infinities and infinitesimals generically  imply all other conditions.
 
 The occurrence of phasor addition and the relations $\  a=(a_1,\cdots, a_n)\in ({\overline{\mathbb{C}}_{as}})^n$, $ \ \sum\limits_i a_i=0$, $\  \|a\|^2 <1$ in Equation~\ref{phasor}  might be of interest. We shall call it the {\it  perfect destructive  interference condition}.

 We shall now prove that the perfect destructive interference condition also holds  for  the transition probability of normal matrices. In fact, let $\ U\in M_n(\overline{\mathbb K})$ be a normal matrix, i.e., $\ UU^*=U^*U$. We may write $\ U = A+iB$ where  $\ A=\frac{1}{2}(U+U^*)$ and $\ B=\frac{1}{2i}(U-U^*)$ are selfadjoint commuting matrices. Let $\ M$ be a unitary matrix such that $ MAM^*=D_1$ is a diagonal matrix, $\ \lambda\in Spec(A)$ and $\ W= V_{\lambda}$. Since $\ W$ is generated by orthogonal free vectors,   it is in fact generated by some of the columns of $\ M$,  it is a free  $\ \overline{\mathbb{K}}-$module.   Since $\ AB =BA$, it follows that $ \ B$ leaves $\ W$ invariant and  $\ B_{|_W}$  is a symmetric operator.   $\ W$ being  a free  $\ \overline{\mathbb{K}}-$module and $\ B_{|_W}$  being  symmetric, implies that  we may decompose $\ W$  as a direct  sum of eigen modules of $\ B$. Hence, modifying some  columns of $\ M$ if necessary, we may suppose that $ \ MBM^* =D_2$ is also a diagonal matrix and thus $ \ MUM^*=D=diag(\lambda_1,\cdots, \lambda_n)$ is diagonal with  $\ MM^*=Id$. Let $\ u, v$ be orthogonal vectors of $\ {\mathbb K}^n$, then $\ \langle exp(iU)u\mid v\rangle = \langle M^*exp(iD)Mu\mid v\rangle=\langle exp(iD)Mu\mid Mv\rangle=\sum\limits_k exp(i\lambda_k)g_k\overline{f_k}$, with  $\ \sum\limits_k g_k f_k = \langle Mu\mid Mv\rangle =\langle u\mid v\rangle =0$, proving the claim. It follows that the results of Section~\ref{finitedim} also holds for normal matrices.   In \cite{unpublish, transition} the following expression is considered: 
 
 $$\bigg|\frac{1}{\sqrt{6}}\bigg[\exp\bigg(\frac{i}{\varepsilon}\bigg)-2\exp\bigg(\frac{i}{\varepsilon^2}\bigg)+(1+\varepsilon)\exp\bigg(\frac{i}{\sqrt{\varepsilon}}\bigg)\bigg]\bigg|$$
 
 \noindent Since this expression  does not satisfy the perfect destructive  interference condition, it does not represent a transition probability coming from a normal matrix.

 \begin{theorem}
 \label{dominado1}    
 Let $\  f=\sqrt{\sum\limits_i g_i \cdot (f_i\circ\mu_i)^2\ }$, where $ g=(g_1,\cdots, g_n)\in ({\overline{\mathbb{C}}_{as}})^n$ or $\ |g_i|\leq C^2 \ \forall \ i$ for some real constant $\ C>0$, the $\ \mu_i$'s are  pure infinities or  infinitesimals  and the  $\ f_i$'s are a.p.  functions.  Then $\ supp(\nu(f))\neq \emptyset$.
 \end{theorem}
 
 \begin{proof} We have that $\  f=|f|=\sqrt{\bigg|\sum\limits_i g_i \cdot (f_i\circ\mu_i)^2\bigg|\  }\leq C\bigg(\sum\limits_i |f_i\circ\mu_i|\bigg)$. The result follows by  Theorem~\ref{meanval1} and  Theorem~\ref{dominado}.
  \end{proof}
  
  It follows from the previous theorem  that $\bigg|\frac{1}{\sqrt{6}}\bigg[\exp\bigg(\frac{i}{\varepsilon}\bigg)-2\exp\bigg(\frac{i}{\varepsilon^2}\bigg)+(1+\varepsilon)\exp\bigg(\frac{i}{\sqrt{\varepsilon}}\bigg)\bigg]\bigg|$ has generalized transition probabilities and the same   holds for every normal matrix $\ A$ whose eigenvalues  are either pure infinities or infinitesimals. The last assertion follows from Equation~\ref{phasor}, Theorem~\ref{dominado} and Theorem~\ref{dominado1}.  With this, we gave an affirmative answer for the existence of generalized transition probabilities in the finite dimensional case. The transition probability defined by Colombeau-Gsponer exists if and only the support of the generalized transition probabilities has a unique element.

\vspace{0.5cm}

\noindent {\bf{Acknowledgements:}} This paper is part of the second author Ph.D.  thesis done under the guidance of the first author. He is grateful to  IME-USP and the Centro de Ci\^{e}ncias de Balsas-UFMA for their hospitality and the possibility to carry out his Ph.D. as programmed. The first author is grateful to J. Aragona and J.F. Colombeau for previous collaborations resulting in \cite{unpublish} and \cite{transition} and to J.F. Colombeau for introducing him into the subject contained in  \cite{Colomarxiv, sponer}.


\begin{thebibliography}{<50>}

\bibitem{antos}  Antosik, P.,  Mikusinski, J., Sikorski, R., {\it Theory of Distributions: The sequential approach},  Elsevier, 1957.
\bibitem{AraBia}  Aragona, J.,   Biagioni. H.A., {\it An intrinsic definition of the Colombeau algebra of generalized functions}, Analysis Mathematica 17,1991, p. 75-132.
\bibitem{Varna}  Aragona, J.,  Catuogno, P.,   Colombeau,  J.F.,  Juriaans,  S.O.,    Olivera, Ch., {\it  Multiplication of distributions in Mathematical Physics},  V.K. Dobrev eds Lie Theory and its Applications in Physics. Springer Proceedings in Math. and Statistics 191, 2017.
\bibitem{AraFerJu}  Aragona, J.,   Fernandez, R.,   Juriaans, S.O., {\it  A discontinuous Colombeau differential calculus},  Monatsh. Math. 144, 2005,pp. 13-29.

\bibitem{unpublish} Aragona, J., Colombeau, J.F., Juriaans, S.O., {\it Multiplication of distributions and nonperturbative  transition probabilities in QFT}, unpublshed (2017).


\bibitem{transition}  Aragona, J.,  Colombeau, J.F., Catuogno, P.,  Juriaans, S.O., Olivera, C.,   {\it Multiplication of Distributions and Nonperturbative Calculations of Transition Probabilities}, Quantum Theory And Symmetries, 393--401, 2017, Springer.

\bibitem{OJRO} Aragona, J., Fernandez, R., Juriaans, S. O., 
  Oberguggenberger, M., {\it Differential calculus and integration of
    generalized functions over membranes}, Monatshefte 
   f\''{u}r Mathematik (2012).
   
   \bibitem{afj2}
Aragona, J.,  Fernandez, R., Juriaans, S.O., 
\newblock {\it The sharp topology on the full Colombeau algebra of generalized functions}
\newblock  (2005), Integral Transforms and Special Functions vol. 17, Nos. 2-3, February-March 2006, 165-170.

 
 \bibitem{afj1} Aragona, J., Fernandez, R., Juriaans, S.O., {\it Natural topologies on Colombeau algebras},  Topol. Methods Nonlinear Anal. 34(1), 161-180 (2009)


 
\bibitem{agj} Aragona, J., Garcia, A.~R.~G.,  Juriaans, S.~O.,
  {\it Algebraic theory of Colombeau's generalized numbers}, Journal
  of Algebra 384 (2013),  194-211. 


\bibitem{agj1} Aragona, J., Garcia, A.~R.~G.,  Juriaans, S.~O.,
  {\it Generalized solutions of nonlinear parabolic equation with
    generalized functions as initial data}, Nonlinear Analysis 71(2009),
   5187-5207.


\bibitem{AJ}
Aragona, J., Juriaans, S.O., 
\newblock {\it Some Structural Properties of the Topological Ring of Colombeau Generalized Numbers}
\newblock Comm. Alg. 29(5), (2001), 2201-2230.


   
   \bibitem{JJO} Aragona, J., Juriaans, S.~O., Oliveira, O.~R.~B, 
  Scarpal\'{e}zos, D., \emph{Algebraic and geometric theory of the
    topological ring of Colombeau generalized functions},
Proceedings of the Edinburgh Mathematical Society (Series 2) 51 (2008),
545-564.


\bibitem{AS}
Aragona, J.,  Soares, M., 
\newblock {\it An Existence Theorem for an Analytic first order PDE in the Framework of Colombeau's theory}
\newblock  Monatsh. Math. 134, (2001), 9-17.
 
\bibitem{Besi}   Besicovitch, A.S., {\it  Almost Periodic Functions},  Cambridge University Press, 1954.


\bibitem{bhatia}  Bhatia, R., {\it Perturbation bounds for matrix eigenvalues},  SIAM, 2007.
\bibitem{bhatiainf}  Bhatia, R.,  {\it The Hoffman-Wielandt inequality in infinite dimensions}, Proceedings of the Indian Academy of Sciences-Mathematical Sciences, vol.104(3), 483--494, 1994. 

\bibitem{Bia}   Biagioni, H.A.,  {\it A Nonlinear Theory of Generalized Functions} Lecture Notes in Math. 1421. Springer, 1990.

 
 \bibitem{Bohr}   Bohr, H., {\it  Almost Periodic Functions},  Chelsea Pub Company, New York, 1947.
 
\bibitem {chris}   Christyakov, V.V., {\it  The Colombeau generalized nonlinear analysis and the Schwartz linear distribution theory}, J. Math. Sci. 93, 1999, p. 3-40.

\bibitem{Colom1}  Colombeau,  J.F., {\it  A multiplication of distributions},  J. Math. Anal. Appl. 94, 1, 1983, p. 96-115.
\bibitem{Colom2}  Colombeau,  J.F., {\it  A general multiplication of distributions},  Comptes rendus Acad. Sci. Paris 296,1983, pp. 357-360.
\bibitem{Colom3} Colombeau,  J.F., {\it New Generalized Functions and Multiplication of Distributions},  Noth-Holland-Elsevier, 1984. 
\bibitem{Colom4} Colombeau,  J.F., {\it Elementary Introduction to New Generalized Functions},   Noth-Holland-Elsevier, 1985. 
 
 \bibitem{Colom5} Colombeau,  J.F., {\it  Multiplication of Distributions}, Lecture Notes in Math. 1532, SpringerVerlag, Berlin-Heidelberg-New York, 1992.

\bibitem{Colom6} Colombeau,  J.F., {\it Multiplication of distributions},  Bull. Amer. Math. Soc., 23, 2, 1990, p. 251-268.
\bibitem{Colom7} Colombeau,  J.F., {\it  Nonlinear generalized functions.}, Sao Paulo J.  Math. Sci. 7, 2013, 2,pp; 201-239.

\bibitem{Colom8} Colombeau,  J.F., {\it  Mathematical problems on generalized functions and the canonical Hamiltonian formalism},  arXiv.org 0708.3425.


 \bibitem{Colom9} Colombeau, J.F.,
\newblock   New Generalized Functions and Multiplication of Distributions, 
\newblock {\it North Holland}, Amsterdam 1984.


 \bibitem{Colomarxiv}  Colombeau, J.F., {\it Mathematical problems on generalized functions and the canonical Hamiltonian formalism}, arXiv,  preprint arXiv:0708..3425,  2007.

\bibitem{sponer1} Colombeau, J.F.,  Gsponer, A.,  {\it The Heisenberg-Pauli canonical Hamiltonian formalism of quantum field theory in the rigorous mathematical setting of nonlinear generalized functions} (part 1),  arXiv.org 0807.0289. 2008.

 
\bibitem {Egorov}  Egorov, Y., {\it A theory of generalized functions},  Russian Math. Surveys 45, 5, 1990, p. 1-49.

 
 
	
 \bibitem{walter} Garcia, A.R.G.,  Juriaans, S.O.,  Rodrigues, W.,  Silva, J.C.,   {\it Off-Diagonal Condition},  Monatsh. Math., to appear (2022).
 
   \bibitem{spjornal} Garcia, A.R.G., Juriaans, S. O., Oliveira, J., Rodrigues, W. M., {\it  A Non-linear Theory of Generalized Functions}, S\~ao Paulo Journal of Mathematical Sciences, Springer, Published online (2021).
   
   
   
\bibitem{walter1}   Garcia, A.R.G., Juriaans, S. O., Oliveira, J., Rodrigues, W. M., {\it Sets of Uniqueness for holomorphic Nets }, ResearchGate,   DOI: 10.13140/RG.2.2.36425.98405, 2020.

\bibitem{garetto} Garetto, C.,   {\it Topological Structures in Colombeau Algebras: Topological $\tilde{\mathbb C}$-modules and Duality Theory}, Acta Applicandae Mathematica, vol.88(1), 81-123, 2005. 

\bibitem{garetto1} Garetto, C.,  Vernaeve, H.,   {\it {Hilbert  ${\tilde{\mathbb C}}$-modules: structural properties and applications to variational problems}, Transactions of the American Mathematical Society,  vol.363(4), 2047-2090, 2011.


\bibitem{gkos1} Grosser, M.,   Farkas, E.,   Kunzinger, M.,  Steinbauer, R., {\it On the Foundations of Nonlinear Generalized Functions I and II},  Memoirs of the AMS 153, 2001.

\bibitem{gkos2} Grosser, M.,   Farkas, E.,   Kunzinger, M.,  Steinbauer, R.,  {\it Geometric Theory of Generalized Functions with Applications to General Relativity},  Kluwer 2001.


\bibitem{gkos3} Grosser, M., Kunzinger, M., Oberguggenberger, M., Steinbauer, R.,
\newblock {\it Geometric Theory of Generalized Functions with Applications to General Relativity, }
\newblock  Kluwer Acad. Publ. vol 537, 2001.
 
\bibitem{gkos4} Grosser M, Kunzinger M., Steinbauer R., Vickers J.A.,
\newblock {\it  A Global Theory of Algebras of Generalized Functions}, 
\newblock  Adv. Math., 166 (1), (2002), 50-72.



\bibitem {sponer} Gsponer, A., {\it A concise introduction to Colombeau generalized functions and their application in classical electrodynamics},  European J. Phys. 30, 2009, 1, pp. 109-126.

 

\bibitem{hosk}  Hoskins, R.F.,    Sousa Pinto, J., {\it  Distributions, Ultradistributions and other Generalized Functions},  Ellis-Horwood, New York-London, 1994.


\bibitem{hoffman} Hoffman, A.J,  Wielandt, H.W,  {\it The variation of the spectrum of a normal matrix}, Duke Math. J. (20), 1, 37-39, 1953.
 

\bibitem{horman1}  Hormann, G., Konjik, S.,  Kunzinger, M.,  {\it Symplectic modules over Colombeau-generalized numbers}, Comm. Algebra, 42(8), 3558--3577, 2014.


 \bibitem{juroliv} Juriaans, S.O.,   Oliveira, J.,  {\it Fixed Point Theorems for Hypersequences and the Foundation of Generalized Differential Geometry I: The Simplified Algebra}, arXiv preprint arXiv:2205.00114, 2022, DOI: 10.13140/RG.2.2.32732.05761.
         
\bibitem{inter}  Juriaans, S.O., {\it A Fixed Point Theorem for Internal Sets},  Research Gate, DOI: 10.13140/RG.2.2.16367.30884, 2022.

\bibitem{kastler}  Kastler, D., {\it  Introduction a l'Electrodynamique Quantique.}, Dunod, Paris, 1961.
 

   
 
 \bibitem{ku}
Kunzinger, M.,
\newblock {\it Lie Transformation Groups in Colombeau algebras,}
\newblock  Doctoral Thesis, University of Viena, 1996.


\bibitem{ko}
Kunzinger, M., Oberguggenberger M.,
\newblock {\it Characterization of Colombeau Generalized Functions by their Point Value,}
\newblock  Math. Nachr. 203, (1999), 147-157.


 \bibitem{ku1}  Kunzinger, M., Oberguggenberger, M., {\it  Characterization of Colombeau generalized functions by their point values},  Math. Nachr. 203, 147-157 (1999)


\bibitem{ku2} Kunzinger, M.,  Steinbauer, R., {\it A rigorous solution concept for geodesics and geodesic deviation equations in impulsive gravitational waves}, J. Math. Phys. 40 (1999), 1479-1489.


\bibitem{ku3} Farkas, E., Grosser, M, Kunzinger, M.,  Steinbauer, R., {\it On the foundations of nonlinear generalized functions I, II}, Mem. Am. Soc. 153(729) (2001).

\bibitem{ku4} M, Kunzinger, M.,  Steinbauer, R., {\it Generalized Pseudo-Riemannian Geometry}, T.A.M.S., Vol. 354, Number 10, 4179-4199.




\bibitem{mayer} Mayerhofer, E. , {\it On Lorentz geometry in algebras of generalized functions},  Proc. R. Soc. Edinb. Sect. A, Math., 138(4):843-871, 2008.

 
\bibitem{Nedel}  Nedeljkov,  M.,  Pilipovic, S.,  Scarpalezos. D., {\it The Linear Theory of Colombeau Generalized Functions},  Pitman Research Notes in Math. 1998.

 
\bibitem{mog}
 Oberguggenberger, M.
\newblock {\it Multiplication of Distributions and Applications to Partial Differential Equations},
\newblock  Pitman, 1992.

\bibitem{opd}
Oberguggenberger, M.,  Pilipovic S.,  Scarpalezos D.,
\newblock {\it Local Properties of Colombeau Generalized Functions},
\newblock  Math. Nachr. 256, (2003), 88-99.



\bibitem{op}
Oberguggenberger,  M.,  Pilipovic S., Valmorin V., 
\newblock {\it Global Representatives of Colombeau holomorphic Generalized Functions}, 
\newblock  preprint,  2005.
 
 
\bibitem{OV}   Oberguggenberger, M., Vernaeve, H.,  {\it  Internal sets and internal functions in Colombeau theory},  J. Math. Anal. Appl. 341, 649-659 (2008).


\bibitem{Ober}  Oberguggenberger, M., {\it Multiplication of Distributions and Partial Differential Equations},  Pitman Research Notes in Math. 259, Longman, Harlow, 1992.
\bibitem{Penrose}   Penrose, R., {\it  The Road of Reality. A Complete Guide to the Laws of the Universe}  Jonathan Cape, 2004.


 
  \bibitem{rob} Robinson, A., {\it Non-standard Analysis}, North-Holland, Amsterdam, (1966).
  

\bibitem{ros1} Rosinger, E.~E.,  {\it Distributions and nonlinear
     partial differential equations}, Lect. Notes Math. 684, Springer,
   Berlin (1978).

\bibitem{ros2} Rosinger, E.~E., {\it Nonlinear partial differential
     equations. Sequential and weak solutions}, North Holland,
   Amesterdam (1980).

\bibitem{ros3} Rosinger, E.~E.,  {\it Generalized solutions of
     nonlinear partial differential equations}, North Holland,
   Amesterdan (1987).
   
\bibitem{ros4} Rosinger, E.~E.,  {\it Which are the maximal ideals}, Prespacetime Journal, vol. 2,  Issue 2, pg.  144-156, 2011. 

 \bibitem{ros5} Rosinger, E.~E., \emph{Non-linear partial differential
     equations. An algebraic view of generalized solutions}, North
   Holland, Amestedam (1990).
   


 
\bibitem{S1}  Scarpalezos, D.,  {\it Topologies dans les espaces de nouvelles fonctions generalis\'ees de Colombeau. ${{\overline{\mathbb{C} }}}$ topologiques}, 
  Universit\'e Paris 7, 1993

\bibitem{S2}  Scarpalezos, D.,  {\it Colombeau's generalized functions: topological structures micro local properties. A simplified point of view,}  CNRS-URA212,
 Universit\'e Paris 7, 1993.
 
 \bibitem{S3} Scarpalezos, D., \emph{Colombeau's generalized functions: Topological structures; Microlocal properties. A simplified point of
view. Part I}, Bull. Cl. Sci. Math. Nat. Sci. Math 121 (25) (2000), 89-114.


\bibitem{schmudgen} Schm\"{u}dgen, K., {\it Unbounded Self-adjoint Operators on Hilbert Space}, Graduate Texts in Mathematics 265, Springer Verlag, 2012.

\bibitem{Schwartznote}   Schwartz, L., {\it  Sur l 'impossibilit\'e de la multiplication des distributions},  Comptes Rendus Acad. Sci. Paris 239, 1954, p. 847-848.

\bibitem{Schwartzlivre}  Schwartz, L., {\it   Th\'eorie des Distributions}, Hermann, Paris, 1966.
\bibitem{SteinVi}   Steinbauer, S., Vickers, J.A.,   {\it  The use of Generalized Functions and Distributions in General Relativity},  Class. Quant. Grav. 23, pp. R91-113, 2006.
\bibitem{Vi} Vickers, J.A., {\it Distributional geometry in general relativity},  J. Geom. Phys. 62,3,2012, p. 692-705.


  \bibitem{tod} Todorov, T.D., {\it Pointwise Value  and Fundamental Theorem in the Algebra of Asymptotic Functions}, arXiv (2006). 

\bibitem{todver} Todorov, T.D., Vernaev, H., {\it Full algebra of generalized functions  and non-standard asymptotic analysis}, Log Anal, (2008) 205-234.



\bibitem{Weinberg}  Weinberg, S. {\it  The Quantum Theory of Fields} Vol. 1, Cambridge, 1995.

 
  
  \bibitem{ver}  Vernaeve, H., {\it  Ideals in the ring of Colombeau generalized numbers}, 
Communications in Algebra, 38, p 2199 - 2228, 2010.

 

\bibitem{ver2}   Vernaeve, H., {\t Banach $\widetilde{\mathbb C}$-algebras}  https://doi.org/10.48550/arxiv.0811.1742}, arXiv, 2008.

		
\bibitem{V1} Vernaeve, H.,  {\it  Optimal embeddings of Distributions into Algebras}, 
Proc. Edinburgh Math. Soc. (2003) 46, 373-378.
		

\bibitem{V2} Vernaeve, H.,  {\it   Isomphisms  of Algebras of Colombeau Generalized Functions}, Monatsh. Math. 162, 225-237 (2011).
 
 \bibitem{V3} Vernaeve, H.,  {\it   Banach $\widetilde{\mathbb{C}}$-algebras}, arXiv, 0811.1742v1[mathFA] 11 Nov (2008).
 

 \bibitem{wolfram} Wolfram, S., {\it A New Kind of Science}, Wolfram Media, Inc., 2002.


\end{thebibliography}
\end{document}